\documentclass[sn-mathphys,Numbered]{sn-jnl}

\usepackage{graphicx}
\usepackage{multirow}
\usepackage{amsmath,amssymb,amsfonts}
\usepackage{amsthm}
\usepackage{mathrsfs}
\usepackage[title]{appendix}
\usepackage[table]{xcolor}
\usepackage{textcomp}
\usepackage{manyfoot}
\usepackage{booktabs}
\usepackage{algorithm}
\usepackage{algorithmicx}
\usepackage{algpseudocode}
\usepackage{listings}

\usepackage{microtype}
\usepackage[caption=false]{subfig}
\usepackage{url}
\usepackage[shortlabels]{enumitem}
\usepackage[export]{adjustbox}
\usepackage{bm}
\usepackage[utf8]{inputenc} 
\usepackage[T1]{fontenc}    
\graphicspath{{./}{fig/}}

\theoremstyle{thmstyleone}
\newtheorem{theorem}{Theorem}

\newtheorem{lemma}[theorem]{Lemma}
\newtheorem{corollary}[theorem]{Corollary}

\theoremstyle{thmstyletwo}

\newtheorem{remark}{Remark}

\theoremstyle{thmstylethree}
\newtheorem{definition}{Definition}
\newtheorem{assumption}{Assumption}

\def\unif{\mathrm{Unif}}

\def\sign{\mathrm{sign}}
\DeclareMathOperator*{\argmin}{arg\,min}

\begin{document}

\title[Curvature-Aware Derivative-Free Optimization]{Curvature-Aware Derivative-Free Optimization}

\author[1]{\fnm{Bumsu} \sur{Kim}}\email{bumsu@math.ucla.edu}
\equalcont{These authors contributed equally to this work.}

\author[2]{\fnm{Daniel} \sur{McKenzie}}\email{dmckenzie@mines.edu}
\equalcont{These authors contributed equally to this work.}

\author[3]{\fnm{HanQin} \sur{Cai}}\email{hqcai@ucf.edu}
\equalcont{These authors contributed equally to this work.}

\author*[4]{\fnm{Wotao} \sur{Yin}}\email{wotao.yin@alibaba-inc.com}
\equalcont{These authors contributed equally to this work.}

\affil[1]{\orgdiv{Department of Mathematics}, \orgname{University of California, Los Angeles}, \orgaddress{ \city{Los Angeles}, \postcode{90095}, \state{CA}, \country{USA}}}

\affil[2]{\orgdiv{Department of Applied Mathematics and Statistics}, \orgname{Colorado School of Mines}, \orgaddress{\city{Golden}, \postcode{80401}, \state{CO}, \country{USA}}}

\affil[3]{\orgdiv{Department of Statistics and Data Science and Department of Computer Science}, \orgname{University of Central Florida}, \orgaddress{\city{Orlando}, \postcode{32816}, \state{FL}, \country{USA}}}

\affil*[4]{\orgdiv{Decision Intelligence Lab, DAMO Academy}, \orgname{Alibaba US}, \orgaddress{\city{Bellevue}, \postcode{98004}, \state{WA},  \country{USA}}}

\abstract{The paper discusses derivative-free optimization (DFO), which involves minimizing a function without access to gradients or directional derivatives, only function evaluations. Classical DFO methods, which mimic gradient-based methods, such as Nelder-Mead and direct search have limited scalability for high-dimensional problems. Zeroth-order methods have been gaining popularity due to the demands of large-scale machine learning applications, and the paper focuses on the selection of the step size $\alpha_k$ in these methods. The proposed approach, called Curvature-Aware Random Search (CARS), uses first- and second-order finite difference approximations to compute a candidate $\alpha_{+}$.
We prove that for strongly convex objective functions, CARS converges linearly provided that the search direction is drawn from a distribution satisfying very mild conditions. We also present a \textbf{C}ubic \textbf{R}egularized variant of CARS, named CARS-CR, which converges in a rate of $\mathcal{O}(k^{-1})$ 
without the assumption of strong convexity.
Numerical experiments show that CARS and CARS-CR match or exceed the state-of-the-arts on benchmark problem sets.}



\maketitle

\section{Introduction}
We consider minimizing a function $f: \mathbb{R}^{d}\to \mathbb{R}$, with only access to function evaluations $f(x)$, and no access to gradients or directional derivatives. This setting is commonly referred to as derivative-free optimization (DFO). DFO has a rich history and has recently gained popularity in various areas such as reinforcement learning \cite{salimans2017evolution,mania2018simple,choromanski2020provably}, hyperparameter tuning \cite{bergstra2012random}, and adversarial attacks on neural network classifiers \cite{chen2017zoo,cai2020zeroth}. In all of these applications, evaluating $f(x)$ is either expensive, time-consuming, or inconvenient, and therefore, it is desirable for DFO algorithms to minimize the number of function evaluations required.

Classical methods for DFO include the Nelder-Mead simplex method \cite{nelder1965simplex}, direct search methods \cite{kolda2003optimization}, and model-based methods \cite{conn2009introduction}. However, these methods tend to scale poorly with the problem dimension $d$, although recent works \cite{cartis2022scalable,cartis2022global,cartis2022dimensionality} have made progress in this direction. Due to the demands of large-scale machine learning applications, \textit{zeroth-order} (ZO) methods for DFO have gained increasing attention \cite{liu2020primer}. ZO methods mimic first-order methods like gradient descent but approximate all derivative information using function queries. At each iteration, the algorithm selects a direction $u_k$ and takes a step $x_{k+1} = x_k + \alpha_k u_k$. While the selection of $u_k$ has been well studied (see \cite{berahas2021theoretical} and references therein), this paper focuses on the selection of $\alpha_k$, allowing for $u_k$ to be either randomly selected or an approximation to the negative gradient (i.e., $u_k \approx -\nabla f(x_k)$).

Intelligently choosing $\alpha_k$ can lead to convergence in fewer iterations, but this gain may be offset by the number of queries it takes. If we compute $u_k \approx -\nabla f(x_k)$, techniques such as backtracking line search from first-order optimization can be employed \cite{berahas2021global}. However, obtaining a sufficiently accurate approximation to $-\nabla f(x_k)$ requires $\Omega(d)$ queries per iteration \cite{berahas2021theoretical}, which is impractical for large $d$. On the other hand, when we take $u_k$ as a random vector, with high probability $u_k$ is almost orthogonal to $-\nabla f(x_k)$. Hence, $\alpha_k$ in \cite{ghadimi2013stochastic,nesterov2017random,bergou2020stochastic} is very small to guarantee descent at every iteration (possibly in expectation). Our approach differs from these methods.

We propose using finite difference approximations to the first and second derivatives of the univariate function $\alpha \mapsto f(x_k +\alpha u_k)$ to compute a candidate $\alpha_{+}$ for $\alpha_k$. Specifically, we set 
\begin{align*}
\alpha_{+} & = \frac{d_r}{\hat{L}h_r},
\end{align*}
where
\begin{align*}
\quad d_{r} &:= \frac{f(x_k+r u_k) - f(x_k-r u_k)}{2r}, \\
h_{r} &:= \frac{f(x_k+r u_k) - 2f(x_k) + f(x_k-r u_k)}{r^2},
\end{align*}
and $\hat{L}$ is a user-specified parameter. Computing $\alpha_+$ requires only three queries per iteration. This simple modification to the well-known Random Search algorithm \cite{ghadimi2013stochastic,nesterov2017random} (which takes $\alpha_k = d_r/\hat{L}$ or similar) can be viewed as an inexact one-dimensional Newton's method at each iteration. When encountering low curvature directions, $h_r$ is small and $\alpha_{+}$ is large, so this $\alpha_{+}$ may occasionally fail to guarantee descent. To remedy this, we combine our step-size rule with a simple safeguarding scheme based on the recently introduced Stochastic Three Point method \cite{bergou2020stochastic}, which guarantees $f(x_{k+1}) \leq f(x_k)$ at every iterate. Importantly, we show that $\alpha_{+}$ is a good candidate, i.e., $f(x_{k+1})$ is significantly smaller than $f(x_k)$ a positive proportion of the time. From this, we can quantify the expected total number of function queries required to reach a target solution accuracy. Because our method is a natural extension of Random Search that incorporates second derivative information, we dub it Curvature-Aware Random Search, or CARS.

In addition to CARS, we propose an extension called CARS-CR (CARS with Cubic Regularization) that modifies the stochastic subspace cubic Newton method \cite{pmlr-v119-hanzely20a} into a zeroth-order method. CARS-CR is essentially CARS with an adaptive parameter $\hat{L}$ and achieves $\mathcal{O}(k^{-1})$ convergence for convex functions.

Our numerical experiments show that both CARS and CARS-CR outperform state-of-the-art algorithms on benchmarks across various problem dimensions, demonstrating efficiency and robustness. Furthermore, our results on adversarial attacks show that CARS can be adapted to different sample distributions of $u_k$. We demonstrate that CARS performs well with a tailored distribution for a particular problem, an adversarial attack on a pre-trained neural network.

\vspace{0.1in}
\noindent\textit{\textbf{Organization.}}\quad
This paper is laid out as follows. In the rest of this section, we fix the notation and discuss prior art. In Section~\ref{sec:CARS}, we introduce the main algorithm, namely Curvature-Aware Random Search (CARS), along with its convergence analysis. 
Section~\ref{section:CARS_CR} extends CARS with Cubic Regularization (CARS-CR) for general convex functions. In Section~\ref{sec: proofs of main results}, we provide mathematical proofs to support our technical claims. Section~\ref{sec:ExperimentalResults} contains extensive numerical experiments that empirically verify our technical claims. Section~\ref{sec:conclusion} concludes the paper.

\subsection{Assumptions and Notation}
\label{sec:Notation} 
In developing and analyzing CARS, we assume that $f$ is a convex and twice continuously differentiable function. We use $g(x) = \nabla f(x)$ and $H(x) = \nabla^2 f(x)$ briefly in the theoretical analysis of Section~\ref{section:convergence of CARS}. For a fixed initial point $x_0$, we define the level-set $\mathcal{Q} = \{x\in\mathbb{R}^d : f(x)\leq f(x_0)\}$, $\|\cdot\|$ as the Euclidean norm, and $f_{\star} := \min_{x\in\mathbb{R}^{d}}f(x)$. We say $x_k$ is an $\varepsilon$-optimal solution if $f(x_k) - f_{\star} \leq \varepsilon$. We use $\mathcal{D}$ to denote a probability distribution on $\mathbb{R}^{d}$. For any measurable set $S \subseteq \mathbb{R}^{d}$ with finite measure, $\mathrm{Unif}(S)$ denotes the uniform distribution over $S$. The unit sphere is written as $\mathbb{S}^{d-1}:= \{u: \|u\| = 1\}\subseteq \mathbb{R}^{d}$, and $e_1,\cdots, e_d$ represent the canonical basis vectors in $\mathbb{R}^{d}$. For two matrices $A$ and $B$, we write $A\preceq B$ if $B-A$ is positive semi-definite.

\begin{definition}
\label{Definition: smoothness}
We say $f$ is $L$-smooth, $L>0$, if $H(x) \preceq L I_d$ for all $x \in \mathcal{Q}$.
\end{definition}

\begin{definition}
\label{Definition: strong convexity}
We say $f$ is $\mu$-strongly convex, $\mu>0$, if $\mu I_d \preceq H(x)$ for all $x \in \mathcal{Q}$. 
\end{definition}

Under strong convexity, $H(z)$ is positive definite for all $ z\in \mathcal{Q}$; hence the following inner product and induced norm are well-defined for all $z \in \mathcal{Q}$:
\begin{align*}
  \langle x, y \rangle _{H(z)} &:= \langle H(z)x, y \rangle  \qquad\textnormal{and}\qquad 
  \|x\|_{H(z)}^2 := \langle x, x\rangle_{H(z)}.  
\end{align*}

Strong convexity also implies the following \cite[Proposition~2]{gower2019rsn}.
\begin{lemma}[$\hat{L}$-Relative Smoothness and $\hat{\mu}$-Relative Convexity]
\label{lemma: Smoothness and convexity}
If $f$ is $\mu$-strongly convex, then $f$ is $\hat{\mu}$-relatively convex and $\hat{L}$-relatively smooth for some $\hat{L} \geq \hat{\mu}>0$, \textit{i.e.} for all $x, y \in \mathcal{Q} $
\begin{equation*}
   \frac{\hat{\mu}}{2}\|x-y\|_{H(y)}^2 \leq  f(x) - f(y) - \langle g(y), x-y \rangle \leq \frac{\hat{L}}{2}\|x-y\|_{H(y)}^2.
\end{equation*}
\end{lemma}

We also make the following regularity assumption on $H$.
\begin{assumption}\label{assumption: Holder continuity of Hessian}
$H$ is $a$-H\"{o}lder continuous for some $a>0$, \textit{i.e.}
\begin{equation}\label{assumption: eq: Holder continuity of H}
    |{u}^{\top} (H(x) - H(y))\, {u}| \leq L_a \|x-y\|^{a}
\end{equation}
for any unit vector ${u}\in \mathbb{S}^{d-1}$ and $x, y \in \mathcal{Q}$.
\end{assumption}
H\"{o}lder continuity reduces to Lipschitz continuity when $a=1$. Assumption~\ref{assumption: Holder continuity of Hessian} can be used to refine the relative smoothness and relative convexity constants in a smaller region. 

\subsection{Prior Art}
\label{sec:PriorArt}

For a comprehensive introduction to DFO we refer the reader to \cite{conn2009introduction} or the more recent survey article \cite{larson2019derivative}. As mentioned above, our interest is in ZO approaches to DFO \cite{liu2020primer}, as these have low per-iteration query complexity (with respect to the dimension of the problem) and have been successfully used in modern machine learning applications, such as adversarial attacks on neural networks \cite{chen2017zoo,liu2018zeroth, cheng2019sign, cai2020zeroth,cai2021zeroth} and reinforcement learning \cite{salimans2017evolution,choromanski2018structured,fazel2018global}. Of particular relevance to this work is ZO algorithms based on {\em line search}:
\begin{subequations}
\begin{align}
    & \mathrm{Sample}\ u_k \ \mathrm{from}\  \mathcal{D}, \nonumber \\
    & \alpha_k \approx \alpha_{\star} = \argmin_{\alpha \in \mathbb{R}}f(x_k + \alpha u_k), \label{eq:line_search}\\
    & x_{k+1} = x_{k} + \alpha_k u_k, \label{eq:generic_step}
\end{align}\label{eq:GenericDFOLS}
\end{subequations}
\hspace{-0.05in}which may be thought of as zeroth-order analogues of  coordinate descent \cite{nesterov2012efficiency}. All of the complexity results discussed below assume {\em noise-free} access to $f(x)$. The noisy case is more complicated, see \cite{jamieson2012query}. 
The first papers to use this scheme were \cite{karmanov1974convergence,karmanov1975convergence}, where convergence is discussed under the assumptions that $u_k$ is a descent direction\footnote{$u_{k}^{\top}\nabla f(x_k) < 0$.} for all $k$ and \eqref{eq:line_search} is solved sufficiently accurately. Assuming \eqref{eq:line_search} is solved {\em exactly}, \cite{mutsenieks1964extremal} proves this scheme finds an $\varepsilon$-optimal solution in $\mathcal{O}(\log(1/\varepsilon))$ iterations when $f$ is a quadratic function (see also \cite{schrack1976optimized} for a discussion of these results in English). In \cite{krutikov1983rate},  $\mathcal{O}(\log(1/\varepsilon))$ iteration complexity was proved assuming access to an approximate line search oracle 
that solves \eqref{eq:line_search} sufficiently accurately, for any strongly convex $f$, as long as $u_k$ are cyclically sampled coordinate vectors.  Similar ideas can be found in \cite{grippo1988global, grippo2007nonmonotone,grippo2015class}. More recently, \cite{stich2013optimization} studied \eqref{eq:GenericDFOLS} under the name {\em Random Pursuit} which assumes access to an approximate line search oracle satisfying either additive ($\alpha_{\star} - \delta \leq \tilde{\alpha} \leq \alpha_{\star} + \delta$) or multiplicative ($(1-\delta)\alpha_{\star} \leq \tilde{\alpha} \leq \alpha_{\star}$ and $\sign(\tilde{\alpha}) = \sign(\alpha_{\star})$) error bounds. They show Random Pursuit finds an $\varepsilon$-optimal solution in $\mathcal{O}(\log(1/\varepsilon))$ (resp. $\mathcal{O}(1/\varepsilon )$) iterations if $f$ is strongly convex (resp. convex). The use of $\mathcal{O}(\cdot)$ above suppresses the dependence of the query complexity on the dimension $d$. In all results stated, the query complexity scales at least linearly with $d$. This is unavoidable in DFO for generic $f$; see \cite{wang2018stochastic, balasubramanian2018zeroth,cai2020zeroth,cai2020scobo,cai2021zeroth,cartis2021scalable} for recent progress in overcoming this. 

We highlight a shortcoming of the aforementioned works: Although they provide essentially optimal bounds on the {\em iteration} complexity, they do not bound the {\em query} complexity. Indeed, the true query complexity will depend on the inner workings of the solver employed to solve \eqref{eq:line_search}. For example, \cite{stich2013optimization} reports each call to the line search oracle requires an average of $4$ function queries when $d \leq 128$ which increases to $7$ when $d = 1024$.  In contrast, CARS requires {\em only three queries} per iteration, independent of $d$. The recently introduced Stochastic Three Point (STP) method  \cite{bergou2020stochastic,bibi2020stochastic} also uses only three queries per iteration. However, STP is not scale invariant, and in practice we find its performance compares poorly against CARS (see Section~\ref{sec:ExperimentalResults}).

\begin{table}[t]
\def\ROWCOLOR{black!10!white}
    \begin{tabular}{l c c c c}
    \toprule
       Algorithm & Strg. Convex & Convex & Queries/Iter &  Line Search Oracle \\
    \midrule
    \rowcolor{\ROWCOLOR}
    \cite{karmanov1975convergence} & $\mathcal{O}(\log(1/\varepsilon))$ & $\mathcal{O}(1/\varepsilon)$ & --- & Yes \\
    \cite{krutikov1983rate} & $\mathcal{O}(\log(1/\varepsilon))$ & --- & --- & Yes \\
    \rowcolor{\ROWCOLOR}
    NDFLS \cite{grippo2015class} & --- & --- & $< \infty$ & No\\
    Random Pursuit \cite{stich2013optimization} & $\mathcal{O}(\log(1/\varepsilon))$ & $\mathcal{O}(1/\varepsilon)$ & 4--7\,(empirical) & Yes \\
    \rowcolor{\ROWCOLOR}
    ZOO-Newton \cite{chen2017zoo} & --- & --- & $3$ & No \\
    Stochastic 3 Points \cite{bergou2020stochastic} & $\mathcal{O}(\log(1/\varepsilon))$ & $\mathcal{O}(1/\varepsilon)$ & 3 & No \\
    \rowcolor{\ROWCOLOR}
    {\bf CARS (proposed)} & $\mathcal{O}(\log(1/\varepsilon))$ & $\mathcal{O}(1/\varepsilon)^{\dagger}$ & $3$ or $4^{\dagger}$ & No \\ 
    \bottomrule
    \end{tabular}
    \caption{Comparison of various line-search based ZO algorithms, all of which use random search directions.  We refer to algorithms without an agreed-upon name by the paper in which it first appeared. If a quantity ({\em e.g.} queries per iteration, convergence rate) is not explicitly computed we denote this with ``---''. Notes: $^{\dagger}$: refers to CARS-CR variant. }
    \label{table: ZO_query_comparison}
\end{table}

We are partially motivated by ZOO-Newton \cite{chen2017zoo}, which is essentially CARS with $\mathcal{D} = \mathrm{Unif}(\{e_1,\cdots, e_d\})$. In \cite{chen2017zoo}, it is demonstrated empirically that ZOO-Newton performs well but no theoretical guarantees are provided. Our convergence guarantees for CARS imply convergence of ZOO-Newton as a special case. Many other works consider adapting Newton's method to the derivative-free setting. However, obtaining an estimate of the $d\times d$ Hessian $\nabla^{2}f(x_k)$ for general ({\em i.e.} unstructured) $f(x)$ is difficult. Thus, one needs to either use $\Omega(d^2)$ queries \cite{fabian1971stochastic} in order to obtain an accurate estimate of $\nabla^2f(x_k)$---far too many for most applications---or use a high-variance approximation to $\nabla^{2}f(x_k)$ \cite{spall2000adaptive,ye2018hessian,glasmachers2020hessian,zhu2019efficient,zhu2020hessian}. CARS sidesteps this dichotomy, as it applies Newton's method to a one-dimensional function. Thus the ``Hessian'' to be estimated is $1 \times 1$. 

\subsection{Main Contributions}
We propose a simple and lightweight zeroth-order algorithm: CARS. To derive convergence rates for CARS we use a novel convergence analysis that hinges on the insight that CARS need only significantly decrease the objective function on a positive proportion of the iterates. Our results allow for a H\"{o}lder continuous Hessian---a weaker assumption than the Lipschitz continuity typically considered in such settings.  We also propose a cubic-regularized variant, CARS-CR. The analysis of CARS-CR extends that of the Stochastic Subspace Newton method \cite{pmlr-v119-hanzely20a} to the zeroth-order setting. The key ingredient is a careful handling of the errors introduced by replacing directional derivatives with their finite difference counterparts. Our theoretical results are corroborated by rigorous benchmarking on two datasets: Mor\'{e}-Garbow-Hillstrom \cite{more1981testing} and CUTEst \cite{gould2015cutest}, which reveal that CARS outperforms existing line-search based ZOO algorithms.  Our paper is accompanied by an open-source implementation of CARS (and CARS-CR), available online at \url{https://github.com/bumsu-kim/CARS}.

\section{Curvature-Aware Random Search}\label{sec:CARS}
Given $u_k$ sampled from $\mathcal{D}$, consider the one-dimensional Taylor expansion:
\begin{equation} \label{eq: T2 Taylor poly of 2nd order}
    T_{2}(\alpha;x_k,u_k) := f(x_k) + \alpha u_k^{\top} g_k + \frac{\alpha^2}{2} u_k^{\top} H_k u_k \approx f(x_k + \alpha u_k).
\end{equation}
CARS selects $\alpha_k \approx \argmin_{\alpha}T_{2}(\alpha;x_k,u_k)$. The exact minimizer  $u_k^{\top} g_k/u_k^{\top} H_k u_k$ depends on unavailable quantities. CARS uses $\alpha_k = d_{r_k}/h_{r_k}$, where $d_{r}$ and $h_{r}$ are finite difference approximations:
\begin{align}
\begin{split}
d_{r_k}(x_k;u_k) &:= \frac{f(x_k+r_k u_k) - f(x_k-r_k u_k)}{2r_k} 
= u_k^{\top}g_k + \mathcal{O}(r_k^2\|u_k\|^2),
\end{split} \label{eq:Compute_d}\\
\begin{split}
h_{r_k}(x_k;u_k) & := \frac{f(x_k+r_k u_k) - 2f(x_k) + f(x_k-r_k u_k)}{r^2} 
            = u_k^{\top}H_k u_k + \mathcal{O}(r_k^2\|u_k\|^2).
\end{split}\label{eq:Compute_h}
\end{align}
(We write $d_{r_k}$ and $h_{r_k}$, in place of $d_{r_k}(x_k;u_k)$ and $h_{r_k}(x_k;u_k)$ when $x_k$ and $u_k$ are clear from context.) Thus each iteration of CARS is a zeroth-order analogue of a single iteration of Newton's method applied to $f$ restricted to the line spanned by $u_k$. As is well-known \cite{nesterov2006cubic}, pure Newton's method may not converge. So, following \cite{gower2019rsn} we add a fixed step-size $1/\hat{L}$ and define:
\begin{equation}
    x_{\mathrm{CARS, }k} = x_k - \frac{d_{r_k}}{\hat{L}h_{r_k}} u_k.
    \label{eq: Define CARS}
\end{equation}

We allow the distribution $\mathcal{D}$ to be iteration dependent, {\em i.e.} $u_k$ can be sampled from $\mathcal{D}_k$. In computing $d_{r_k}(x_k)$ and $h_{r_k}(x_k)$, CARS queries $f$ at the symmetric points $x_k + r_k u_k$ and $x_k - r_k u_k$. 
We extend STP~\cite{bergou2020stochastic} into a {\em safeguarding mechanism} for CARS and 
choose the next iterate
\begin{equation*}
    x_{k+1} = \argmin \{ f(x_{\mathrm{CARS}, k}), f(x_k), f(x_k - r_k u_k), f(x_k + r_k u_k) \},
\end{equation*}
which ensures monotonicity:
$f(x_0)\geq f(x_1) \geq f(x_2) \geq \cdots $. CARS requires two input parameters, $\hat{L}$ and $C$. Ideally, $\hat{L}$ should be the relative smoothness parameter (see Lemma~\ref{lemma: Smoothness and convexity}), although CARS-CR (see Section~\ref{section:CARS_CR}) introduces a mechanism for selecting $\hat{L}$ adaptively. The selection of $C$ is the subject of the next section.

\begin{algorithm}[H]
 \caption{\textbf{C}urvature-\textbf{A}ware \textbf{R}andom \textbf{S}earch (CARS)}
 \label{alg:CARS}
\begin{algorithmic}[1]
  \State \textbf{Input:} $x_0$: initial point; $\hat{L}$: relative smoothness parameter,
  $C$: scale-free sampling radius limit.
  \State Get the oracle $f(x_0)$.
  \For{$k=0$ to $K$}
        \State Sample $u_k$ from $\mathcal{D}_k$.
        \State Set $r_k \leq C/\|u_k\|$. 
        \State Evaluate and store $f(x_{k} \pm r_k u_k)$.
        \State Compute $d_{r_k}$ and $h_{r_k}$ using \eqref{eq:Compute_d} and \eqref{eq:Compute_h}.
        \State Compute $x_{\textrm{CARS}, k} = x_{k} - \frac{d_{r_k}}{\hat{L}h_{r_k}}u_k$.
        \State $x_{k+1} = \argmin\{f(x_{\textrm{CARS}, k}),f(x_k), f(x_{k} - r_k u_k), f(x_{k} + r_k u_k) \}$.
  \EndFor
   \State \textbf{Output:} $x_K$: estimated optimum point.
\end{algorithmic}
\end{algorithm}

\subsection{Convergence Guarantees}\label{section:convergence of CARS}

Before proceeding we list two necessary assumptions on $\mathcal{D}_k$. To describe the assumptions, 
introduce:
\begin{align}
    \eta( g, H ; \mathcal{D}) = \mathbb{E}_{u \sim \mathcal{D}} \left[ \frac{(u^\top g)^2}{(u^{\top} H u)  (g^{\top}H^{-1}g) } \right]. \label{eq:Def of eta}
\end{align}
By Cauchy-Schwarz $\eta(g,H;\mathcal{D}) \leq 1$ for all $g$, $\mathcal{D}$, and positive definite $H$. 
We use $\eta$ to measure the quality of the sampling distribution $\mathcal{D}$ with respect to the Newton vector $H^{-1}g$, and it is exactly 1 when all $u\sim \mathcal{D}$ are parallel to $H^{-1}g$. Our analysis assumes $\eta(g,H;\mathcal{D})$ is bounded away from zero, and this property holds for common choices of $\mathcal{D}$ as shown in Lemma~\ref{lemma: Lower bound eta}. Since 
replacing $(r_k, \mathcal{D}_k)$ by $(\beta^{-1}r_k, \beta\mathcal{D}_k)$, for any $\beta >0$, will not affect CARS, we use the {\em{scale-free sampling radius}}, $r_k\|u_k\|$, 
and define the following constants depending on the H\"{o}lder continuity of $H$:
\begin{align*}
    C_{1,a} = \left(\frac{(a+1)(a+2)}{2^{1/2+a}L_{a}}\right)^{1/(1+a)} \text{ and }\quad C_{2,a} = \left( \frac{(a+1)(a+2)}{4(\sqrt{2}+1)L_{a}} \right)^{1/a}.
\end{align*}
Our analysis requires us to define the following sampling radius limit, $C$, which also depends on the target accuracy $\varepsilon$ and a free parameter $\gamma \in (0,1]$:
\begin{align}\label{eq: Scale-Free Sampling Radius Limit}
    C := \min \{ C_{1,a} (\gamma \sqrt{2\mu \varepsilon})^{1/(1+a)}, C_{2,a} \mu^{1/a} \}.
\end{align}
CARS uses $C$ to choose the sampling radius $r_k\|u_k\|$ {\em after} sampling $u_k$ (see Line 6 of Algorithm~\ref{alg:CARS}). For instance, when $H$ is Lipschitz continuous, this rule gives $r_k \|u_k\| = \mathcal{O}(\varepsilon^{1/4})$. Note that $C$ is scale-invariant, {\em i.e.} replacing $(f, \varepsilon)$ by $(\lambda f, \lambda\varepsilon)$ for any $\lambda>0$ does not change $C$.

\begin{theorem}[Expected descent of CARS]
\label{thm: linear descent of CARS at each step}
Suppose $f$ is $\mu$-strongly convex and its Hessian, $H$, is $a$-H\"{o}lder continuous. Suppose further that $\eta(g_k, H_k; \mathcal{D}_k) \geq \eta_{0} > 0$. Let $\gamma \in (0,1]$ and $\varepsilon$ be the target accuracy. Take the scale-free limit of sampling radius $C$ in  \eqref{eq: Scale-Free Sampling Radius Limit}. Let $x_{\mathrm{CARS}, k}$ be as in \eqref{eq: Define CARS} and let $\mathcal{A}_k$ denote the event: 
\begin{equation} \label{eq: condition on u}
    \gamma \|u_k\| \sqrt{2\mu\varepsilon} \leq |{u}_k^{\top}g_k|.
\end{equation}
Then, 
\begin{equation}
    \mathbb{E}\left[f(x_{\mathrm{CARS}, k}) - f_\star \mid \mathcal{A}_k \right] \leq \left(1 - \eta_{0}\frac{\hat{\mu}}{2\hat{L}}\right)(f(x_k) - f_\star).
    \label{eq: CARS Descent}
\end{equation}
\end{theorem}
In words, by limiting the sampling radius to $C$, and conditioning on $u_k$ being ``good enough'' (\emph{i.e.} $\mathcal{A}_k$ occurs,) we obtain linear descent in  expectation. The proof of Theorem~\ref{thm: linear descent of CARS at each step} can be found in Section~\ref{sec: proofs of main results}. Although $\mathcal{A}_k$ does not occur with probability $1$, we show $\mathcal{A}_k$ occurs for a positive fraction of CARS iterations. When $\mathcal{A}_k$ does not occur, the safeguarding mechanism (Line 10 of Algorithm~\ref{alg:CARS}) still ensures monotonicity: $f(x_{k+1}) \leq f(x_k)$. This reveals the key idea behind CARS: {\em it exploits good search directions $u_k$ when they arise yet is robust against poor search directions.} Carefully quantifying this intuition, we have:

\begin{corollary}[Convergence of CARS]\label{thm:convergence of cars; strongly cvx}
Take the assumptions of Theorem~\ref{thm: linear descent of CARS at each step}. Suppose further that there exists $\gamma \in (0,1]$ such that
\begin{equation}\label{eq: p_gamma defining inequality}
    p_{\gamma} := \inf_{k\geq 0} \mathbb{P}_{u_k\sim\mathcal{D}_k} \left[ |u_k^{\top}g_k| \geq \gamma\|u_k\|\|g_k\|\right] >0
\end{equation}
for all $k\geq 0$, and use $\gamma$ to define $C$ in \eqref{eq: Scale-Free Sampling Radius Limit}. 
Then, Algorithm~\ref{alg:CARS} converges linearly. More specifically,
for any 
\begin{align*}
    K \geq \frac{2\hat{L}}{\eta_0 p_{\gamma}\hat{\mu}} \log\left(\frac{f(x_0) - f_\star}{\varepsilon} \right),
\end{align*}
we have $\mathbb{E}[f(x_K)] - f_\star \leq \varepsilon$.
\end{corollary}
\noindent The additional assumption on $\mathcal{D}_k$ {\em i.e.} the existence of $\gamma$, is very mild, and is discussed in Sec.~\ref{subsec: sampling distribution}.

\subsection{Further Results on the Sampling Distribution}\label{subsec: sampling distribution}
The speed of convergence of CARS depends crucially on the lower bounds $\eta_0$ and $p_{\gamma}$ (see \eqref{eq:Def of eta} and \eqref{eq: p_gamma defining inequality}). The following Lemma computes $\eta_0$ for several commonly used distributions.

\begin{lemma}
\label{lemma: Lower bound eta}
\begin{enumerate}
    \item (Isotropic distributions) When
    \begin{align*}
    \mathcal{D} =  \unif(\mathbb{S}^{d-1}), \textnormal{~~}\unif(\{e_1, \cdots, e_d\}), \textnormal{~~} \mathcal{N}(0, I_d), \textnormal{~~or~~} \unif(\{\pm 1\}^d),
    \end{align*}
    we have $\eta(g,H;\mathcal{D}) \geq {\mu}/({dL})$. The distributions in the above equation are uniform on sphere, coordinate directions, Gaussian, and Rademacher, respectively.
    \item (Approximate gradient direction) If $\mathcal{D}$ satisfies
    \begin{align}\label{eq: D approximates gradient}
        \mathbb{E}_{u\sim \mathcal{D}}\left[ \left(\frac{|u^\top g|}{\|u\|\|g\|}^2\right) \right] \geq \beta >0
    \end{align}
    for some $\beta>0$,
    then $\eta(g,H;\mathcal{D}) \geq {\beta \mu}/{L}$.
    \item (Newton direction) When $u$ is parallel to $H^{-1}g$ with probability 1, we have $\eta(g,H;\mathcal{D}) = 1$. 
\end{enumerate}
\end{lemma}
\begin{proof}
Since $u^{\top} H u \leq L\|u\|^2$ and $ g^{\top}H^{-1}g \leq \mu^{-1}\|g\|^2 $, 
\begin{align}\label{eq: eta with curvature term and angle btw gradient}
    \eta(g,H;\mathcal{D}) \geq \frac{\mu}{L} \,\, \mathbb{E}_{u\sim \mathcal{D}}\left[ \left(\frac{|{u}^\top {g}|}{\|u\|\|g\|}\right)^2 \right].
\end{align}

\begin{enumerate}
    \item When $\mathcal{D} = \mathcal{N}(0,I_d) $ or $\unif(\mathbb{S}^{d-1})$, we can replace $g$ by the standard basis vector $e_1$ by symmetry, and it immediately follows that $\eta(g,H;\mathcal{D}) \geq \mu/(dL)$.
    When $\mathcal{D} = \unif(\{e_1, \cdots, e_d\})$,
    \begin{align*}
         \mathbb{E}_{u\sim \mathcal{D}}\left[ \left(\frac{|{u}^\top {g}|}{\|u\|\|g\|}\right)^2 \right]
         = \frac{1}{d}\sum_{i=1}^{d}|g_i|^2 / \|g\|^2 = \frac{1}{d}
    \end{align*}
    and when $\mathcal{D} = \unif(\{\pm 1\}^d )$, 
    \begin{align*}
         \mathbb{E}_{u\sim \mathcal{D}}\left[ \left(\frac{|{u}^\top {g}|}{\|u\|\|g\|}\right)^2 \right]
         = \frac{1}{2^d}\sum_{u \in \{\pm 1\}^d} \frac{\sum_{i=1}^{d}|g_i|^2 + \sum_{i\neq j}u_i u_j g_i g_j}{d \|g\|^2}
         = \frac{1}{d}.
    \end{align*}
    Hence,  again from \eqref{eq: eta with curvature term and angle btw gradient}, we have the same lower bound $\mu/(dL)$.
    
    \item When \eqref{eq: D approximates gradient} holds, \eqref{eq: eta with curvature term and angle btw gradient} provides the lower bound $\eta(g,H; \mathcal{D}) \geq \beta \mu / L$.
    In particular, when $u$ is parallel to $g$ ({\em i.e. gradient direction}) with probability $p$, then $\eta \geq p\mu/L$.
    
    \item When $u$ is the Newton direction, {\em i.e.} $u$ is parallel to $H^{-1}g$ with probability 1, $u^{\top}g = u^{\top}Hu = g^{\top}H^{-1}g$, and so $\eta(g,H;\mathcal{D}) = 1$. 
\end{enumerate}
    This finishes the proof. 
\end{proof}

Lemma~\ref{lemma: Lower bound eta} suggests that assuming $\eta(g_k, H_k; \mathcal{D}_k) \geq \eta_{0} > 0$ for all $k\geq 0$ is reasonable in practice. Note Case 3 yields the best possible $\eta$, as $\eta\leq1$ by Cauchy-Schwarz. The next Lemma suggests that assuming $p_{\gamma} > 0$ is also reasonable in practice.

\begin{lemma}[Estimation and Lower Bounds of $p_{\gamma}$ for Various Distributions]\label{lemma: p_gamma lower bounds} ~
\begin{enumerate}
    \item (Uniform on sphere and Gaussian) When $\mathcal{D} = \mathcal{N}(0, I_d)$ or $\unif(\mathbb{S}^{d-1})$ we have
    \begin{align}\label{eq: p_gamma bound for normal and uniform sphere}
        \mathbb{P}_{u\sim\mathcal{D}} \left[ |u^{\top}g| \geq \gamma \|u\|\|g\|\right] = I_{1-\gamma^2}\left(\frac{d-1}{2}, \frac{1}{2} \right). 
    \end{align}
    In particular, for $d\geq 2$, $\mathbb{P}_{u\sim\mathcal{D}} \left[ |u^{\top}g| \geq \|u\|\|g\|/\sqrt{d} \right] \geq 0.315603$. 
    \item (Random coordinate direction) When $\mathcal{D} = \unif(\{e_1, \cdots, e_d\})$ we have 
    $$\mathbb{P}_{u\sim\mathcal{D}} \left[ |u^{\top}g| \geq \|u\|\|g\|/\sqrt{d}\right] \geq 1/d.$$
    
\end{enumerate}
\end{lemma}

\begin{proof}
\begin{enumerate}
    \item First note that we can assume $\|u\|=1$ in \eqref{eq: p_gamma bound for normal and uniform sphere}, and thus we only need to consider the case $\mathcal{D} = \unif{(\mathbb{S}^{d-1})}$.
    In this case, $\mathcal{D}$ is invariant under rotation so we can take $g = e_1$ and
    \begin{align*}
        \mathbb{P}_{u\sim\mathcal{D}} \left[ |u^{\top}g| \geq \gamma \|u\|\|g\|\right] & = \mathbb{P}[ |u_1| \geq \gamma ] 
         = I_{1-\gamma^2}\left(\frac{d-1}{2}, \frac{1}{2} \right)
    \end{align*}
    where $I$ is the regularized incomplete Beta function as in \cite[Theorem~2.3]{cai2020scobo}.
    In particular, when $\gamma = 1/\sqrt{d}$, the function $d\mapsto I_{1-1/d}\left(\frac{d-1}{2}, \frac{1}{2} \right)$ is decreasing for $d\geq 2$ and bounded below by 0. Thus $p_{\gamma} \geq \lim_{d\rightarrow \infty} I_{1-1/d}\left(\frac{d-1}{2}, \frac{1}{2} \right) = 0.315603\cdots$.
    
    \item When $\mathcal{D} = \unif{\{e_1,\cdots,e_d\}}$, 
    \begin{align*}
        \mathbb{P}_{u\sim\mathcal{D}} \left[ |u^{\top}g| \geq \gamma \|u\|\|g\|\right]
         = \frac{1}{d} \sum_{i=1}^{d} 1_{|g_i| \geq \gamma\|g\|}.
    \end{align*}
    Recall that $\|g\|^2 = \sum_{i}|g_i|^2$. Hence,  we have $\max_{i} |g_i| \geq \|g\|/\sqrt{d}$, which implies $ \mathbb{P}_{u\sim\mathcal{D}} \left[ |u^{\top}g| \geq \|u\|\|g\|/\sqrt{d}\right] \geq 1/d$. Note that this bound is tight; the equality holds when, for example, $g = (1, 0, \cdots, 0)$. 
\end{enumerate}
\noindent This finishes the proof.
\end{proof}

When $\gamma$ is small enough and $\mathcal{D}_k$ approximates the gradient or Newton direction close enough, both $\eta_{\mathcal{D}_k}$ and $p_\gamma$ do not depend on $d$, leading to dimension independent convergence rates. So, CARS can be combined with other derivative-free techniques that estimate the gradient (or Newton direction)---at the cost of two additional function queries per iteration CARS will choose an approximately optimal step-size in this computed direction. Our analysis easily extends to such combined methods, and we sketch how to do so for the widely used \cite{nesterov2017random,salimans2017evolution,choromanski2018structured,fazel2018global} variance-reduced Nesterov-Spokoiny gradient estimate:
\begin{equation}
    \tilde{g}_k := \frac{1}{m} \sum_{i=1}^m d_r(x_k;u_k)u_k \approx g_k.
\end{equation}
For simplicity, we assume access to exact directional derivatives (as in \cite{nesterov2017random}).

\begin{corollary}
   Let $f$ be $\mu$-strongly convex and $H$ be $a$-H\"{o}lder continuous. Suppose, at each step, $u_k$ is generated by first sampling $v_1,\cdots, v_m$ from Gaussian distribution $\mathcal{N}(0, I_d)$ and defining:
   \begin{align*}
        u_k = \frac{1}{m}\sum_{j=1}^{m} (g_k^{\top}v_j) v_j.
    \end{align*}
    Then CARS (Algorithm~\ref{alg:CARS}) finds $x_K$ with $\mathbb{E}[f(x_K)] - f_\star \leq \varepsilon$ if 
    \begin{align*}
    K \geq \frac{2\hat{L}L(m+d+1)}{\hat{\mu}\mu mp_{\gamma}} \log\left(\frac{f(x_0) - f_\star}{\varepsilon} \right).
\end{align*}
\end{corollary}
\begin{proof}
    From Lemma~\ref{lemma: Lower bound eta} (part 2) we obtain:
    \begin{align*}
        \eta_\mathcal{D} \geq \frac{\mu m}{L(m+d+1)}.
    \end{align*}
    Combining this with Corollary~\ref{thm:convergence of cars; strongly cvx} yields the claim.
\end{proof}

\section{CARS with Cubic Regularization for General Convex Functions}
\label{section:CARS_CR}
Here, we adopt cubic regularization \cite{nesterov2006cubic, pmlr-v119-hanzely20a}, a technique to achieve global convergence of a second-order method for convex functions, in CARS and prove convergence.
We drop strong convexity and assume only $L$-smoothness. We assume Lipschitz continuity of the Hessian (\emph{i.e.} $a=1$ in Assumption~\ref{assumption: Holder continuity of Hessian}) and let $M = L_1$ be the Lipschitz constant.
Instead of \eqref{eq: T2 Taylor poly of 2nd order}, we now use 
\begin{equation}
    P(\alpha; d, h) := d\alpha + \frac{1}{2}h\alpha^2 + \frac{M}{6}|\alpha|^3,
\end{equation}
with the exact derivatives $P(\,\cdot\,; d_0, h_0)$ and the finite difference approximations $P(\,\cdot\,; \pm d_{r_k}, h_{r_k})$.
The method of Stochastic Subspace Cubic Newton (SSCN) \cite{pmlr-v119-hanzely20a} takes exact derivatives and uses the following inequality  \cite[Lemma~2.3]{pmlr-v119-hanzely20a}
\begin{equation}\label{eq: cubic regularization bound}
    f(x_{k} + \alpha u_k) \leq f(x_k) + P(\alpha; d_0(x_k; u_k), h_0(x_k; u_k))
\end{equation}
to derive the algorithm $x_{k+1} = x_k + \hat{\alpha}_{k} u_k$, where $\hat{\alpha}_{k} = \argmin_{\alpha} P(\alpha; d_0, h_0)$. We propose using $\alpha_{k}^{\pm} = \argmin_{\alpha}P(\alpha; \pm d_{r_k}, h_{r_k})$ in place of $\hat{\alpha}_k$. By solving $P^{\prime}(\alpha; \pm d_{r_k}, h_{r_k}) = 0$ we obtain
\begin{align*}
    \alpha_k^{\pm} = -\frac{\pm 2d_{r_k}}{h_{r_k} + \sqrt{h_{r_k}^2 + 2M|d_{r_k}|}}.
\end{align*}
This step-size equals $-\frac{\pm d_{r_k}}{h_{r_k}\hat{L}_k}$ with 
\begin{equation} \label{eq: L_k for CARS-CR}
    \hat{L}_k = \frac{1}{2} +\sqrt{\frac{1}{4} + \frac{M|d_{r_k}|}{2h_{r_k}^2}},
\end{equation}
so it is just CARS with this varying relative smoothness constant. 
 We formalize this as Algorithm~\ref{alg:CARS-CR}.

\begin{algorithm}
 \caption{CARS with \textbf{C}ubic \textbf{R}egularization (CARS-CR)}
 \label{alg:CARS-CR}
\begin{algorithmic}[1]
  \State \textbf{Input:}  $\varepsilon$: target accuracy; $x_0$: initial point; $r_0$: initial sampling radius; $M$: Lipschitz constant of Hessian. 
  \State Get the oracle $f(x_0)$.
  \For{$k=0$ to $K$}
        \State Sample $u_k$ from $\mathcal{D}_k$.
        \State Set $r_k \leq \rho\sqrt{\varepsilon}/\sqrt{k+2}$ where $\rho = R/\sqrt{2B}$ as defined in Theorem~\ref{thm: Convergence of CARS-CR}.
        \State Evaluate and store $f(x_{k} \pm r_k u_k)$.
        \State Compute $d_{r_k}$ and $h_{r_k}$ using \eqref{eq:Compute_d} and \eqref{eq:Compute_h}.
        \State Compute $\hat{L}_k$ using \eqref{eq: L_k for CARS-CR}.
        \State Compute $x_{\textrm{CR}\pm, k} = x_{k} \pm \frac{d_{r_k}}{\hat{L}_kh_{r_k}}u_k$.
        \State $x_{k+1} = \argmin\{f(x_{\textrm{CR}+, k}), f(x_{\textrm{CR}-, k}),f(x_k), f(x_{k} - r_k u_k), f(x_{k} + r_k u_k) \}$.
  \EndFor
   \State \textbf{Output:} $x_K$: estimated optimum point.
\end{algorithmic}
\end{algorithm}

To analyze CARS-CR (Algorithm~\ref{alg:CARS-CR}), we make a boundedness assumption.
\begin{definition}
\label{def:R-bounded}
Recall that $\mathcal{Q} = \{x \in \mathbb{R}^d : f(x) \leq f(x_0)\}$. We say $f$ has an {\em $\mathcal{R}$-bounded level set} if the diameter of $\mathcal{Q}$ is $\mathcal{R} < \infty$.
\end{definition}
Without loss of generality, we may assume the distribution is normalized ({\em i.e.} $\|u\|=1$ w.p. $1$.) This is because we only need to bound the scale-free sampling radius $r_k\|u_k\|$, as before.
To ensure that the finite difference error is insignificant, we need the sampling radius small enough. However, for a more concise analysis, it is helpful to have an upper bound, which can be chosen arbitrarily. Let $R>0$ be an upper bound of $r_k$ for all $k\geq0$. Note that any $r_k$ selected by CARS-CR automatically satisfies $r_k \leq R$ (see line 5 of Algorithm~\ref{alg:CARS-CR}). Using this notation, we get:

\begin{lemma}
[Finite difference error bound for the minimum of $P$]
\label{cor: P finite difference error bound}
    Let $P(\,\cdot\,) = P(\,\cdot\,;d_0, h_0)$. Then for any $0 \leq r_k \leq R$,
    \begin{equation}
        \min(|P(\hat{\alpha}_k) - P(\alpha_k^+)|,\, |P(\hat{\alpha}_k) - P(\alpha_k^-)|) \leq \frac{2B}{R^2}r_k^2,
        \label{eq: P_0_perturbation}
    \end{equation}
    where $B = \max(LR^2, MR^3, f(x_0) - f_\star)$.
\end{lemma}

If the sampling distribution is isotropic in expectation, \textit{i.e.} it satisfies  
$ 
    \mathbb{E}\left[ {u_ku_k^{\top}} \right] = \frac{1}{d}I_d, 
$, we get the following descent lemma:
\begin{theorem}[Expected descent of CARS-CR]
\label{thm: Expected descent of CARS-CR}
Suppose $f$ is convex, $L$-smooth, and has $M$-Lipschitz Hessian.
If $\mathcal{D}_k$ is isotropic in expectation, 
then with Algorithm~\ref{alg:CARS-CR}, we have
 \begin{equation}\label{eq: CR expected bound wrt z}
    \mathbb{E}\left[f(x_{k+1})\mid x_k \right] \leq  \left(1-\frac{1}{d}\right)f(x_k) + \frac{1}{d}f(x_k + z)
     + \frac{L}{2d}\|z\|^2 + \frac{M}{6d}\|z\|^3  + \frac{2B}{R^2}r_k^2
 \end{equation}
 for any $z \in \mathbb{R}^d$.
\end{theorem}

Finally, with decreasing $r_k$ as given in Algorithm~\ref{alg:CARS-CR}, we obtain the $\mathcal{O}(k^{-1})$ convergence rate for CARS-CR. 

\begin{theorem}[Convergence of CARS-CR]\label{thm: Convergence of CARS-CR}
Take the assumptions of Theorem~\ref{thm: Expected descent of CARS-CR}, and further assume $f$ has an $\mathcal{R}$-bounded level set. Set $r_k \leq \frac{\rho\sqrt{\varepsilon}}{\sqrt{k+2}}$ where $\rho = \frac{R}{\sqrt{2B}}$. 
Then, with Algorithm~\ref{alg:CARS-CR}, we have
\begin{equation} 
\begin{split}
    \mathbb{E}[f(x_K)] - f_\star \leq  
                &~\frac{ s^s (f(x_0)-f_\star) (1+ \log(K+2)) }{(K/d)^{s+1}}
                + \frac{e^{s/K} (s+1)^2 L\mathcal{R}^2  }{2s(K/d)} \cr
                &~+ \frac{e^{(s-1)/K} (s+1)^3 M \mathcal{R}^3  }{6(s-1)(K/d)^2}
                + \frac{e^{2(s+1)/K} }{s+1}\varepsilon
\end{split}
\end{equation}
for any $s > 1$.
That is, for any $0<p<1$ there exists $C_p > 0$ such that  $\mathbb{E}[f(x_K)] - f_\star \leq \varepsilon$ if
\begin{equation} \label{eq: CR conv K depending on p}
    K  \geq  C_p d \max \left\{ \frac{L\mathcal{R}^2}{\varepsilon} , \sqrt{\frac{M\mathcal{R}^3}{\varepsilon}} , \left(\frac{f(x_0)-f_\star}{\varepsilon}\right)^{p} \right\}.
\end{equation}
\end{theorem}

\section{Proofs}
\label{sec: proofs of main results}

Here we collect the proofs of the results of Sections \ref{section:convergence of CARS} and \ref{section:CARS_CR}, and state and prove some auxiliary lemmas needed in the proofs of the main results. We begin with a lemma quantifying the expected descent given access to exact derivatives.

\subsection{Proofs for Results in Section~\ref{section:convergence of CARS}}
\begin{lemma}[Expected descent of CARS with exact derivatives]
\label{lemma: CARS-ED descent}
Let $u_k \sim \mathcal{D}_k$ and $x_{\mathrm{ED}, k}$ be the CARS step with exact derivatives
\begin{equation}
    x_{\mathrm{ED}, k} = x_k - \frac{u_k^\top g_k}{\hat{L} u_k^\top H_k u_k}u_k.
\end{equation}
Then letting $\eta_k = \eta(g_k, H_k; \mathcal{D}_k)$,
\begin{equation}
  \mathbb{E}\left[f(x_{\mathrm{ED},k}) \mid x_k \right] - f_{\star}\leq  \left(1 - \eta_k \frac{\hat{\mu}}{\hat{L}}\right)(f(x_k) - f_\star).
\end{equation}
\end{lemma}

\begin{remark}
Lemma~\ref{lemma: CARS-ED descent} is similar to \cite[Corollary~1]{gower2019rsn} and \cite[Corollary~1 part (ii)]{kozak2021stochastic}. However, Lemma~\ref{lemma: CARS-ED descent} allows for more general sampling distributions $\mathcal{D}$. 
\end{remark} 

\begin{proof}
From $\hat{\mu}$-relative strong convexity we have
\begin{align}
    f_\star - f(x_k) & \geq \langle g_k, x_\star - x_k \rangle + \frac{\hat{\mu}}{2} \| x_\star - x_k \|_{H_k}^2 \geq -\frac{1}{2\hat{\mu}} \|g_k\|_{H_k^{-1}}^2,
    \label{eq:Rearrange}
\end{align}
where the second inequality follows by taking $x = x_\star - x_k$ and $c = \hat{\mu}$ in the following general inequality \cite[Lemma~9]{gower2019rsn}:
\begin{equation*}
\argmin_{x \in \mathbb{R}^d} \,\, \langle g, x \rangle + \frac{c}{2} \|x\|_H^2 = - \frac{1}{c}H^{-1}g \quad \text{ if $H\succ 0$ and $c>0$}.
\end{equation*}
Rearranging \eqref{eq:Rearrange} yields
    $ - \|g_k\|_{H_k^{-1}}^2 \leq 2\hat{\mu} (f_\star - f(x_k))$.
Let $M_k := \frac{u_k u_k^{\top}}{u_k^{\top} H_k u_k}$. Then, from $\hat{L}$-relative smoothness and \cite[Lemma~5]{gower2019rsn}, 
\begin{equation}
     f(x_{\mathrm{ED}, k}) \leq  f(x_k) - \frac{1}{2\hat{L}} \|g_k\|^2_{M_k} = f(x_k) - \frac{1}{2\hat{L}} \frac{\langle u_k u_k^{\top} g_k, g_k\rangle }{u_k^\top H_k u_k}
= f(x_k) - \frac{1}{2\hat{L}} \frac{(u_k^\top g_k)^2}{u_k^{\top} H_k u_k}.
    \label{eq:TakeCondExp}
\end{equation}
Now let $\mathbb{E}_k[\cdot] := \mathbb{E}[\cdot | x_k]$ and take the conditional expectation of both sides of \eqref{eq:TakeCondExp}:
\begin{align*}
    \mathbb{E}_k\left[f(x_{\mathrm{ED}})\right]
    &\leq f(x_k) - \frac{1}{2\hat{L}} \mathbb{E}_{k}\left[  \frac{(u_{k}^\top g_k)^2}{u_{k}^{\top} H_k u_{k}} \right] \\
    & = f(x_k) - \frac{\eta(g_k, H_k; \mathcal{D}_k)}{2\hat{L}}\|g_k\|^2_{H_k^{-1}} \\ 
    & \leq f(x_k) - \eta_{k}\frac{\hat{\mu}}{\hat{L}}(f(x_k) - f_\star)  
\end{align*}
Subtracting $f_\star$ from both sides yields the desired result.
\end{proof}

\begin{proof}[Proof of Theorem~\ref{thm: linear descent of CARS at each step}]
In this proof, for notational convenience let $d_{0} = g_k^\top {u}_k$ for the first-order directional derivative, and $h_{0} = {u}_k^\top H_k {u}_k$ for the second-order, and denote $r_k$ by $r$. From the definition of $\hat{L}$-relative smoothness, how much we progress at each step can easily be described by a quadratic function $q(t)$: 
\[
f(x_k) - f(x_k+tu_k) \geq q(t) := -d_{0}t - \frac{1}{2}\hat{L}h_{0} t^2.
\]
As in the exact derivatives case, the maximizer of $q$ is $t_\star = -d_{0}/ (\hat{L} h_{0})$, with corresponding maximum $q(t_\star) = d_{0}^2/(2 \hat{L} h_{0}) = \|g_k\|_{M_k}/(2\hat{L})$, where  $M_k := \frac{u_k u_k^{\top}}{u_k^{\top} H_k u_k}$ as before. Recall that $x_{\mathrm{CARS},k} = x_k-d_{r}/(\hat{L}h_{r}) u_k$. Our goal is to show that the finite difference estimate $t_{r} := -d_{r}/(\hat{L}h_{r})$ approximates $t_\star$ well enough so that $q(t_{r}) \geq q(t_\star)/2$.
Observe that if
\begin{equation}
 |t_{r} / t_\star - 1| \leq \sqrt{1-c} \iff |t_{r}- t_\star|^2 \leq (1-c)t_\star^2 
 \label{eq:c_bound}
\end{equation}
holds for some $0 <c<1$, then by completing the square in $q(t)$:
\begin{equation*}
    q(t_r) = -\frac{\hat{L}h_0}{2}(t_r -t_\star)^2 + q(t_\star) \geq -(1-c)q(t_\star) + q(t_\star) = cq(t_\star).
\end{equation*}
Because we want to show $q(t_r) \geq q(t_\star)/2$, it suffices to show \eqref{eq:c_bound} holds for $c=1/2$, {\em i.e.},
\begin{equation}\label{eq: proof goal -- cars str cvx}
\left|\frac{t_r}{t_\star} -1\right| = \left|\frac{d_r /d_0 }{ h_r / h_0} -1\right| \leq \sqrt{1 - \frac{1}{2}} = \frac{1}{\sqrt{2}}.
\end{equation}
To prove \eqref{eq: proof goal -- cars str cvx}, we further bound the left-hand side by the two separate (relative) finite difference errors.
Let $e_d$ and $e_h$ be the absolute errors in estimating $d_{0}$ and $h_{0}$, respectively, {\em i.e.} $e_d = |d_{0} - d_r|$ and $e_h = |h_{0} - h_r|$. 
Then, when $e_h < h_0$, which will be shown shortly,
\begin{equation*}
    \left|\frac{d_r /d_{0} }{ h_r / h_{0}}-1\right| 
    = \left| \frac{ -\frac{d_0-d_r}{d_0} + \frac{h_0-h_r}{h_0}}{1 - \frac{h_0-h_r}{h_0}}\right|
    \leq \frac{e_d/|d_{0}|+e_h/h_{0}}{1-e_h/h_{0}},
\end{equation*}
and thus, for \eqref{eq: proof goal -- cars str cvx} we only need to prove 
\begin{equation}\label{eq: sufficient condition for mu}
    \frac{e_d}{|d_{0}|} + \left(1+\frac{1}{\sqrt{2}}\right)\frac{e_h}{h_{0}} \leq \frac{1}{\sqrt{2}}.
\end{equation}

Now we bound $e_d$ and $e_h$ using Taylor's theorem and Assumption~\ref{assumption: Holder continuity of Hessian}. Because we have 
\begin{align} \label{eq: finite difference with Taylor's theorem}
    f(x_k \pm r u_{k}) = f(x_k) \pm r g_k^{\top}u_{k} + r^2 \int_0^1 (1-t) u_{k}^{\top}H(x_k \pm tr u_{k})u_{k} \,dt,
\end{align}
we get the following representation for the error of the first-order directional derivative:
\begin{align*}
    d_r - d_{0} & = \frac{f(x_k + r u_{k})-f(x_k - r u_{k})}{2r} - g_k^{\top}u_{k} \\
    &= \frac{r}{2}\int_0^{1} (1-t) u_{k}^{\top}\left[H(x_k + tr u_{k}) - H(x_k - tr u_{k})\right]u_{k} \,dt.
\end{align*}
By Assumption~\ref{assumption: Holder continuity of Hessian}, 
    $ \left|u_{k}^{\top} \left[H(x_k + tr u_{k}) - H(x_k - tr u_{k})\right] u_{k}\right| \leq L_a (2tr)^a\|u_{k}\|^{a+2} $
and therefore,
\begin{equation}\label{eq: d_mu - a bound}
    e_d = |d_r - d_{0}| \leq   2^{a-1} L_ar^{a+1} \|u_{k}\|^{a+2} \int_0^1 (1-t)t^a\,dt 
    = \left(\frac{r\|u_k\|}{C_{1,a}}\right)^{1+a}\frac{\|u_k\|}{2\sqrt{2}}.
\end{equation}
Similarly, for the second-order directional derivative, 
\begin{equation} \label{eq: h_mu - b bound}
    e_h = |h_r - h_{0}| \leq  2 L_a r^{a} \|u_{k}\|^{a+2} \int_0^1 (1-t)t^a\,dt 
     = \left(\frac{r\|u_k\|}{C_{2,a}}\right)^{a}\frac{\|u_k\|^2}{2\sqrt{2}+2}
\end{equation}
We see that $r\|u_k\| \leq C = \min\{C_{1,a}(\gamma \sqrt{2\mu \varepsilon})^{1/(1+a)}, C_{2,a}\mu^{1/a} \}$ implies two separate bounds
    \begin{equation}\label{eq: bound on e_d and e_h}
        e_d \leq \frac{\gamma\sqrt{\mu\varepsilon}\|u_k\|}{2} \stackrel{(a)}{\leq} \frac{|d_0|}{2\sqrt{2}} \quad \text{ and } \quad e_h \leq \frac{\mu\|u_k\|^2}{2\sqrt{2} + 2} \stackrel{(b)}{\leq} \frac{h_0}{2\sqrt{2}+2},
    \end{equation}
    where (a) holds assuming $\mathcal{A}_k$ occurs and (b) follows from strong convexity:
    \begin{equation}
        h_0 = u_k^{\top}H_k u_k \geq \mu.
    \end{equation}
     As \eqref{eq: bound on e_d and e_h} implies \eqref{eq: sufficient condition for mu} we have proved the theorem. 
\end{proof}

We now are ready to prove the convergence of CARS (Algorithm~\ref{alg:CARS}).
\begin{proof}[Proof of Corollary~\ref{thm:convergence of cars; strongly cvx}]
From strong convexity 
we have
\begin{equation*}
    f_\star - f(x) \geq \langle g(x), x_{\star} - x \rangle + \frac{\mu}{2}\|x_{\star} - x\|^2 
    \geq -\frac{1}{2\mu}\|g(x)\|^2,
\end{equation*}
for any $x \in \mathbb{R}^d$, where the second inequality comes from
\begin{equation*}
    \argmin_{x\in\mathbb{R}^d} \langle g, x \rangle + \frac{c}{2}\|x\|^2 = -\frac{1}{c}g.
\end{equation*}
Thus $\|g(x)\|^2 \geq 2\mu(f(x)-f_{\star})$. Taking expectation on both sides $\mathbb{E}[\|g(x_k)\|^2] \geq 2\mu (\mathbb{E}[f(x_k)]-f_{\star})$.

If $\|g(x_k)\|^2 \leq 2\mu\varepsilon$ at the $k$-th step with $k\leq K$, then $f(x_K) - f_{\star} \leq \varepsilon$ as $f(x_k)$ is monotonically decreasing by definition (See line 9 of Algorithm~\ref{alg:CARS}.) Thus we need only consider the case where $\|g(x_k)\|^2 > 2\mu\varepsilon$ for all $k<K$; because if the expectation of $f(x_K)$ conditioned on this event is less than or equal to $f_{\star} + \varepsilon$, then the total expectation is also bounded by the same value.

The key of the proof is that $\mathcal{A}_k$ occurs with probability at least $p_{\gamma} > 0$. Indeed, we have  $|u_k^{\top}g_k| \geq \gamma\|u_k\|\|g_k\|$ with probability at least $p_{\gamma}$, and since $\|g_k\| > \sqrt{2\mu\varepsilon}$, 
\begin{align*}
    \mathbb{P}[\mathcal{A}_k] \geq
     \mathbb{P}\left[|u_{k}^{\top} g_{k}| \geq \gamma \|u_k\| \|g_k\| \geq  \gamma \|u_k\| \sqrt{2\mu\varepsilon}\right] \geq p_{\gamma}.
\end{align*}
If $\mathcal{A}_k$ occurs then by Theorem~\ref{thm: linear descent of CARS at each step}, we get 
\begin{equation*}
    \mathbb{E}[f(x_{k+1})|\mathcal{A}_k] - f_{\star} \leq \left(1-\eta_{\mathcal{D}}\frac{\hat{\mu}}{2\hat{L}}\right)(f(x_k)-f_{\star}).
\end{equation*}
If $\mathcal{A}_k$ does not occur then, as CARS is non-increasing, $f(x_{k+1}) \leq f(x_k)$. Thus
\begin{align*}
 \mathbb{E}\left[f(x_{k+1}) \mid x_k \right] - f_{\star} & = \mathbb{E}[f(x_{k+1})- f_{\star}|\mathcal{A}_k]\mathbb{P}[\mathcal{A}_k] + \mathbb{E}[f(x_{k+1})- f_{\star}|\mathcal{A}_k^{c}]\mathbb{P}[\mathcal{A}_k^{c}]\\
 & \leq \left(1-\eta_{\mathcal{D}}\frac{\hat{\mu}}{2\hat{L}}\right)(f(x_k)-f_{\star})\mathbb{P}[\mathcal{A}_k] + \left(f(x_k) - f_{\star}\right)\left(1 - \mathbb{P}[\mathcal{A}_k]\right) \\
 & = \left(1-\eta_{\mathcal{D}}      
     \mathbb{P}[\mathcal{A}_k]\frac{\hat{\mu}}{2\hat{L}}\right)(f(x_k)-f_{\star}) \\
&\leq \left(1-\eta_{\mathcal{D}} p_{\gamma}\frac{\hat{\mu}}{2\hat{L}}\right)(f(x_k)-f_{\star}) \\
 \Rightarrow  \mathbb{E}[f(x_{k+1})] - f_{\star} & \leq \left(1-\eta_{\mathcal{D}}p_{\gamma}\frac{\hat{\mu}}{2\hat{L}}\right)^{k+1}(f(x_0)-f_{\star}),
\end{align*}
whence solving for $K$ in
\begin{equation}
    \left(1-\eta_{\mathcal{D}}p_{\gamma}\frac{\hat{\mu}}{2\hat{L}}\right)^{K}(f(x_0)-f_{\star}) \leq \varepsilon
\end{equation}
completes the proof. 
\end{proof}

\subsection{Proofs for Results in Section~\ref{section:CARS_CR}}
Recall that:
\begin{equation*}
    P(\alpha; d, h) := d\alpha + \frac{1}{2}h\alpha^2 + \frac{M}{6}|\alpha|^3
\end{equation*}
(we write $P(\alpha)$ in place of $P(\alpha; d, h)$ when $d$ and $h$ are clear from context.) Define the map $\phi: \mathbb{R}\times\mathbb{R}_{\geq 0} \rightarrow \mathbb{R}$:
\begin{equation*}
    \phi(d, h) := \argmin_{\alpha}P(\alpha; d, h).
\end{equation*}
Note that not only $h_0 \geq 0$, but also $h_{r_k} \geq 0$ due to the convexity of $f$:
\begin{equation*}
    h_{r_k}(x_k; u_k) = \frac{2}{r_k^2}\left(\frac{f(x_k + r_k u_k) + f(x_k - r_k u_k)}{2} - f(x_k)\right) \geq 0.
\end{equation*}
Then $\hat{\alpha}_k = \phi(d_0, h_0)$ and $\alpha_k^{\pm} = \phi(\pm d_{r_k}, h_{r_k})$ by their definition.
Along the way, we have useful identities for $\phi$:
\begin{align}\label{eq: closed form expression for phi}
    \phi(d, h) = \frac{\sign(d)}{M}\left(h - \sqrt{h^2 + 2M|d|}\right) = \frac{-2d}{h + \sqrt{h^2 + 2M|d|}},
\end{align}
and
\begin{align}\label{eq: alpha Useful_ID}
        \frac{M}{2}|\alpha_{\mathrm{min}}|\alpha_{\mathrm{min}} = -d - h\alpha_{\mathrm{min}}.    
\end{align}
Note that \eqref{eq: closed form expression for phi} shows that $\phi$ is well-defined.
We first describe the perturbation of $\phi$, and how $P$ behaves near its minimum.
\begin{lemma}[Perturbation of $\phi$]\label{lemma: perturbation of phi}
    Let $d, d' \in \mathbb{R}$ have the same sign and $h, h' \geq 0$.
    Defining $S = \sqrt{h^2 + 2M|d|}$ and $S' = \sqrt{(h')^2 + 2M|d'|}$,
    \begin{equation}
        |\phi(d, h) - \phi(d', h')|
        \leq \frac{|h-h'|}{M} + \frac{2|d-d'|}{S + S'}.
    \end{equation}
\end{lemma}
\begin{proof}
    Because $d$ and $d^{\prime}$ have the same sign, from \eqref{eq: closed form expression for phi}, we obtain that $\phi(d,h)$ and $\phi(d', h')$ have the same sign and so $|\phi(d,h) -\phi(d', h')| = \frac{1}{M}|S - S' - (h - h')|$, whence
    \begin{align*}
        |\phi(d,h) -\phi(d', h')| &= \frac{1}{M}\left|\left(S - S'\right)\frac{S + S'}{S + S'} - (h - h')\right| = \frac{1}{M}\left|\frac{S^2 - (S')^2}{S + S'} - (h - h')\right|\\
        & = \frac{1}{M} \left|
        \frac{(h-h')(h+h')}{S+S'} + \frac{2M(|d|-|d'|)}{S+S'} - (h-h')
        \right|\\
        & \leq  \frac{1}{M} \left(1 - \frac{h+h'}{S+S'}\right)|h-h'| + \frac{2|d-d'|}{S+S'} 
        ~~ \leq \frac{|h-h'|}{M} + \frac{2|d-d'|}{S+S'},
    \end{align*}
    where the last inequality comes from that $0 \leq h + h' \leq S + S'$. 
\end{proof}
We now analyze the effect of perturbations to $\alpha_{\mathrm{min}}$ on $P(\alpha)$, under the assumption that the perturbed value of $\alpha$ has the same sign as $\alpha_{\mathrm{min}}$.
\begin{lemma}[Perturbation of $P(\alpha)$ near minimum]
\label{lemma: perturbation of P near min}
    Let $d \in \mathbb{R}$ and $h \geq 0$. Define $\alpha_{\mathrm{min}} = \phi(d, h)$, and let $\alpha^{\prime} \in \mathbb{R}$ have $\sign(\alpha') = \sign(\alpha_{\mathrm{min}})$. Then
    \begin{equation}\label{eq: Perturbation of P near min}
       0 \leq P(\alpha';d,h) - P(\alpha_{\mathrm{min}};d,h)
        \leq \frac{1}{2}(\alpha_{\mathrm{min}} - \alpha')^2 ( h + M|\alpha_{\mathrm{min}}| + \frac{M}{3}|\alpha_{\mathrm{min}} - \alpha'|).
    \end{equation}
\end{lemma}
\begin{proof}
    Let $\sigma = \sign(\alpha_{\mathrm{min}}) = \sign(\alpha')$. We write $P(\alpha_{\mathrm{min}})$, resp. $P(\alpha)$, for $P(\alpha_{\mathrm{min}};d,h)$, resp. $P(\alpha';d,h)$. Then,
    \begin{align*}
        &\quad~P(\alpha') - P(\alpha_{\mathrm{min}}) \\
        & = d (\alpha' - \alpha_{\mathrm{min}}) + \frac{h}{2}(\alpha' - \alpha_{\mathrm{min}})(\alpha' + \alpha_{\mathrm{min}}) + \frac{\sigma M}{6}(\alpha' - \alpha_{\mathrm{min}})((\alpha')^2 + \alpha_{\mathrm{min}}^2 + \alpha_{\mathrm{min}}\alpha') \\
        & = (\alpha' - \alpha_{\mathrm{min}})\left(d + \frac{h}{2}(\alpha' + \alpha_{\mathrm{min}}) + \frac{\sigma M}{6} ((\alpha')^2 + \alpha_{\mathrm{min}}^2 + \alpha_{\mathrm{min}} \alpha')\right).
    \end{align*}
    Using \eqref{eq: alpha Useful_ID}, we get
    \begin{align*}
        P(\alpha') - P(\alpha_{\mathrm{min}}) & = (\alpha' - \alpha_{\mathrm{min}})\left(
        \frac{h}{2}(\alpha' - \alpha_{\mathrm{min}}) + \frac{\sigma M}{6}((\alpha')^2 - 2\alpha_{\mathrm{min}}^2 + \alpha_{\mathrm{min}} \alpha')
        \right) \\
        & = \frac{1}{2}(\alpha'- \alpha_{\mathrm{min}})^2 \left(h + \frac{M}{3}|\alpha' + 2\alpha_{\mathrm{min}}| \right) \\
        &\leq \frac{1}{2}(\alpha' - \alpha_{\mathrm{min}})^2 \left(h + M|\alpha_{\mathrm{min}}| + \frac{M}{3}|\alpha' - \alpha_{\mathrm{min}}| \right).
    \end{align*}
    Noting that $P(\alpha') - P(\alpha_{\mathrm{min}})\geq 0$ as $\alpha_{\mathrm{min}}$ minimizes $P(\alpha)$ we obtain the desired statement.
\end{proof}

From \eqref{eq: closed form expression for phi} we see that if $\sign(d_{r_k}) = \sign(d_0)$ then $\sign(\hat{\alpha}_k) = \sign(\alpha^{+}_k)$, whence we may use the perturbation bounds of Lemmas \ref{lemma: perturbation of phi} and \ref{lemma: perturbation of P near min}. If $\sign(d_{r_k}) = -\sign(d_0)$ then $\sign(\hat{\alpha}_k) = \sign(\alpha^{-}_k)$ and the conclusions of Lemmas \ref{lemma: perturbation of phi} and \ref{lemma: perturbation of P near min} still apply. We conclude that at least one of $\alpha_k^+$ and $\alpha_k^-$ is a good approximation for $\hat{\alpha}_k$, and formalize this as Lemma~\ref{cor: P finite difference error bound}.
\begin{proof}[Proof of Lemma~\ref{cor: P finite difference error bound}]
    First, assume that $\sign(d_0) = \sign(d_{r_k})$, 
    so $\sign(\hat{\alpha}_k) = \sign(\alpha_k^{+})$ by \eqref{eq: closed form expression for phi}. Thus, by Lemma~\ref{lemma: perturbation of P near min},
    \begin{equation}\label{eq: P bound, to be used}
        |P(\hat{\alpha}_k) - P(\alpha_k^+)| \leq 
        \frac{1}{2}(\alpha_k^+ - \hat{\alpha}_k)^2\left(h_0 + M |\hat{\alpha_k}| + \frac{M}{3}|\alpha_k^+ - \hat{\alpha}_k|\right).
    \end{equation}
    Since $h_0 = u_k ^{\top} H_k u_k \leq L$, it only remains to find appropriate bounds for $|\alpha_k^+ - \hat{\alpha}_k|$ and $\hat{\alpha}_k$. For notational convenience, define $S_r := \sqrt{h_r^2 + 2M|d_r|}$ for $r\geq 0$. As $f$ is convex we know that $|d_0| \leq \|g_k\| \leq \sqrt{2L(f(x_k)-f_\star)}$, see \cite[Prop. B.3]{bertsekas1997nonlinear} and so
    \begin{align*}
        M|\hat{\alpha}_k|
        & \stackrel{\eqref{eq: closed form expression for phi}} = |S_0 - h_0| \leq \sqrt{S_0^2 - h_0^2}
        = \sqrt{2M|d_0|}
        \leq \sqrt{2M\|g_k\|} \\
        &\leq \sqrt{2M\sqrt{2L(f(x_k)-f_\star)}} 
         = \sqrt{\frac{2}{R^3}\left(MR^3\right)\sqrt{\frac{2}{R^2}(LR^2)(f(x_k)-f_\star)}} 
        \leq \frac{2^{3/4} B}{R^2},
    \end{align*}
    using the definition of $B = \max(LR^2, MR^3, f(x_k)-f_\star)$. Defining the finite difference errors $e_k^d = d_{r_k} - d_{0} $ and $e_k^h = h_{r_k} - h_{0}$,
    Lemma~\ref{lemma: perturbation of phi} implies
    \begin{equation}
        |\hat{\alpha}_k - \alpha_k^+| \leq \frac{|e_k^h|}{M} + \frac{2|e_k^d|}{S_0 + S_{r_k}}.
        \label{eq:Wednesday_1}
    \end{equation}
    As $\|u_k\|=1$ and $H$ is assumed Lipschitz continuous ({\em i.e.} $a=1$), from \eqref{eq: h_mu - b bound}, we have $|e_k^h| \leq \frac{Mr_k}{3}$ and the first term on the right-hand side of \eqref{eq:Wednesday_1} is bounded by $\frac{r_k}{3}$. Appealing to \eqref{eq: d_mu - a bound} we obtain $|e_k^d| \leq \frac{Mr_k^2}{6}$. We use this and the fact that $\sign(d_0) = \sign(d_k)$ to bound the second term on the right-hand side of \eqref{eq:Wednesday_1}:
    \begin{align*}
        \frac{2|e_k^d|}{S_0 + S_{r_k}} = \frac{2 |d_k - d_0|}{S_{0}+S_{r_k}}
        & \leq \frac{2\left(\sqrt{|d_0|} + \sqrt{|d_{k}|}\right) \left|\sqrt{|d_0|}-\sqrt{|d_{k}|}\right|}{\sqrt{2M|d_0|} + \sqrt{2M|d_{k}|}} \\
        &= \frac{2\left|\sqrt{|d_0|}-\sqrt{|d_{k}|}\right|}{\sqrt{2M}} \leq 
        \frac{2\, \sqrt{|d_0 - d_{k}|}}{\sqrt{2M}}  
         \leq 2\,\,\sqrt{\frac{1}{2M}\frac{Mr_k^2}{6}} 
         = \frac{r_k}{\sqrt{3}} .
    \end{align*}
    This provides a nice bound independent of $L$, $M$, and $R$; $ |\hat{\alpha}_k - \hat{\alpha}_k| \leq (1/3 + 1/\sqrt{3}) r_k < r_k.$  Combining everything with \eqref{eq: P bound, to be used}, we get
    \begin{align*}
        \left| P_0(\hat{\alpha}_k) - P_0(\alpha^{+}_k) \right| 
        & < \frac{1}{2}r_k^2\left( L + \frac{2^{3/4} B}{R^2} + \frac{M}{3}r_k\right) 
        ~~\leq \frac{1}{2}r_k^2\left(\frac{B}{R^2} + \frac{2^{3/4} B}{R^2} + \frac{B}{3R^2} \right) \\
        & \leq \frac{Br_k^2}{R^2}\left(\frac{2}{3} + \frac{1}{2^{1/4}}\right) 
        ~~\leq \frac{2B}{R^2}r_k^2.
    \end{align*}
    If $\sign(d_0) = -\sign(d_{r_k})$ then $\sign(\hat{\alpha}_k) = \sign(\alpha^{-}_k)$, again by \eqref{eq: closed form expression for phi}. Lemma~\ref{lemma: perturbation of phi} and Lemma~\ref{lemma: perturbation of P near min} now yield
    \begin{align}
        & |\hat{\alpha}_k - \alpha_k^-| \leq \frac{|e_k^h|}{M} + \frac{2|d_0 - (-d_{r_k})|}{S_0 + S_{r_k}} \label{eq:Wednesday_2} \\
        & |P(\hat{\alpha}_k) - P(\alpha_k^-)| \leq 
        \frac{1}{2}(\alpha_k^- - \hat{\alpha}_k)^2\left(h_0 + M |\hat{\alpha_k}| + \frac{M}{3}|\alpha_k^- - \hat{\alpha}_k|\right). \nonumber
    \end{align}
    The first term in \eqref{eq:Wednesday_2} can be bounded as before. Because $|d_0 + d_{r_k}| \leq |d_0 - d_{r_k}|\leq |e^d_k|$ as $d_0$ and $d_{r_k}$ have opposite signs, the second term in \eqref{eq:Wednesday_2} is bounded by $r_{k}/\sqrt{3}$ as before. Following the proof of the $\sign(d_0) = \sign(d_{r_k})$ case we conclude that,
    \begin{equation*}
        \left| P_0(\hat{\alpha}_k) - P_0(\alpha^{-}_k) \right| \leq \frac{2B}{R^2}r_k^2,
    \end{equation*}
    thus proving the theorem.  
\end{proof}

\begin{proof}[Proof of Theorem~\ref{thm: Expected descent of CARS-CR}]
    First, fix $u_k \in \mathbb{R}^d$ drawn from $\mathcal{D}_k$. Then, for $\sigma = -\sign(d_0(x_k;u_k))$ and any $z \in \mathbb{R}^d$,
    \begin{align*}
        &\quad~f(x_{k+1}) - f(x_k) \\
        &\leq  f(x_k + \alpha_k^{\sigma} u) -f(x_k) 
        ~~ \leq P(\alpha_k^{\sigma}; d_0(x_k; u_k), h_0(x_k; u_k))  & \text{\small (Eq.~\eqref{eq: cubic regularization bound})}\\
        & \leq P(\hat{\alpha}_k; d_0, h_0) + \frac{2B}{R^2}r_k^2 & \text{(\small Lemma~\ref{cor: P finite difference error bound})}\\
        & \leq P(u_k^{\top}z; d_0, h_0) + \frac{2B}{R^2}r_k^2  & \text{\small (minimality of $\hat{\alpha}_k$)}\\
        & = (z^{\top}u_k)(u_k^{\top}g_k) + \frac{1}{2}(z^{\top}u_k)(u_k^{\top}H_k u_k)(u_k^{\top}z) + \frac{M}{6}|u_k^{\top}z|^3 + \frac{2B}{R^2}r_k^2 &~
    \end{align*}
    holds. 
    Now taking the expectation and using the isotropy condition:
    \begin{align*}
        \mathbb{E}\left[ f(x_{k+1}) \mid x_k \right] - f(x_k) 
        & \leq \frac{1}{d}z^{\top}g_k +
        \frac{1}{2} z^{\top}\mathbb{E} \left[ u_k u_k^{\top} H_k u_k u_k^{\top} \right]z + \frac{M}{6}\mathbb{E}\left[|u_k^{\top} z |^3\right] + \frac{2B}{R^2}r_k^2.
    \end{align*}
    Note that the expectations above satisfy
    $ \frac{1}{2} z^{\top}\mathbb{E} \left[ u_k u_k^{\top} H_k u_k u_k^{\top} \right]z \leq \frac{1}{2}z^{\top}\mathbb{E}\left[Lu_k u_k^{\top}\right]z = \frac{L}{2d}\|z\|^2$
    and
    $
        \mathbb{E}\left[|u_k^{\top} z |^3\right]
         \leq \mathbb{E}\left[|u_k^{\top} z |^2\right]\|z\| = \frac{1}{d}\|z\|^3,
    $
    respectively.
    Therefore, 
    \begin{align}\label{eq: CR expected descent in z involving g and H}
        \mathbb{E}\left[ f(x_{k+1}) \mid x_k \right] - f(x_k) 
        & \leq \frac{1}{d}z^{\top}g_k
        + \frac{L}{2d}\|z\|^2 + \frac{M}{6d}\|z\|^3 + \frac{2B}{R^2}r_k^2.
    \end{align}
    Finally, using convexity of $f$, namely $f(x_k + z)-f(x_k) \geq z^{\top}g_k$, we obtain \eqref{eq: CR expected bound wrt z}.
\end{proof}

\begin{proof}[Proof of Theorem~\ref{thm: Convergence of CARS-CR}]
Let $\delta(x)$ denote the optimality gap $f(x) - f_\star$, and $\delta_k := \mathbb{E}[\delta(x_k)]$. Since Algorithm~\ref{alg:CARS-CR} has non-increasing $\delta_k$, we may assume $\delta_0 > \varepsilon$. Note that $\delta$ is convex. Letting $x_{\star}$ be any fixed minimizer ({\em i.e.} $f(x_{\star}) = f_{\star}$), we note that $\delta(x_\star) = 0$. For any $t_k \in (0, 1)$,  setting $z = t_k(x_\star - x_k)$ in Theorem~\ref{thm: Expected descent of CARS-CR} and defining $\Delta_k = \|x_\star - x_k\|$ yields
\begin{align*}
    &\quad~\mathbb{E}\left[f(x_{k+1}) \,|\, x_k \right] - f_{\star} \cr
    &\leq  (1-\frac{1}{d})f(x_k) + \frac{1}{d}f((1-t_k)x_k + t_kx_{\star}) - f_{\star} 
     + \frac{L}{2d} t_k^2 \Delta_k^2 + \frac{M}{6d} t_k^3 \Delta_k^3 + \frac{2B}{R^2} r_k^2 
\end{align*}
and
\begin{align}
     & \delta_{k+1} \leq (1-\frac{1}{d})f(x_k) +\frac{1-t_k}{d}f(x_k) + \frac{t_k}{d}f_{\star} - f_{\star}
     + \frac{L}{2d} t_k^2 \Delta_k^2 + \frac{M}{6d} t_k^3 \Delta_k^3 + \frac{2B}{R^2} r_k^2 \label{eq:Thm3-convexity} \\
     & \delta_{k+1} \leq (1 - \frac{1}{d} + \frac{1}{d} - \frac{t_k}{d})f(x_k) - (1 - \frac{t_k}{d})f_{\star} + \frac{L}{2d} t_k^2 \Delta_k^2 + \frac{M}{6d} t_k^3 \Delta_k^3 + \frac{2B}{R^2} r_k^2 \label{eq:March-14}\\
    & \delta_{k+1} \leq (1-\frac{t_k}{d})\delta_k +  \frac{L}{2d} t_k^2 \Delta_k^2 + \frac{M}{6d} t_k^3 \Delta_k^3 + \frac{2B}{R^2} r_k^2 \label{eq: CR conv before telescoping},
\end{align}
where in \eqref{eq:Thm3-convexity} we use the convexity of $f$, in \eqref{eq:March-14} we use $f(x_{\star}) = f_{\star}$,  and in \eqref{eq: CR conv before telescoping} we use the definition of $\delta_k$. We adopt an auxiliary sequence $\{\beta_k\}$ to make \eqref{eq: CR conv before telescoping} telescoping.
Let $s>1$, and define $\gamma_k = k^s$ and $\beta_k = \beta_0 + \sum_{j=1}^{k}\gamma_j$ with $\beta_0 = s^s d^{s+1}/(s+1)$, then $t_k = d\frac{\gamma_{k+1}}{\beta_{k+1}} \in (0, 1)$, and $1 - \frac{t_k}{d} = \frac{\beta_k}{\beta_{k+1}}$. We further note that: 
\begin{equation}\label{eq: beta_k bounds}
    \frac{k^{s+1}}{s+1} \leq \beta_0 + \int_{1}^{k}\frac{1}{x^s} dx \leq \beta_k \leq \beta_0 + \int_{2}^{k+1}\frac{1}{x^s}dx = \beta_0 + \frac{(k+1)^{s+1}}{s+1}
\end{equation}
Then by multiplying $\beta_{k+1}$ on both sides of \eqref{eq: CR conv before telescoping}, we get
\begin{align*}
    \beta_{k+1}\delta_{k+1} \leq \beta_k \delta_k + \frac{Ld}{2}\frac{ \gamma_{k+1}^2}{\beta_{k+1}}\Delta_k^2
                + \frac{Md^2}{6} \frac{ \gamma_{k+1}^3}{\beta_{k+1}^2}\Delta_k^3
                + \frac{2B}{R^2} \beta_{k+1} r_k^2,
\end{align*}
and summing up from $k=0$ to $K-1$, we have
\begin{align}\label{eq: CR conv delta K bound as Sum}
    \delta_K \leq \frac{\beta_0}{\beta_K}\delta_0 + \frac{Ld}{2\beta_K} \sum_{k=1}^{K} \frac{\gamma_k^2}{\beta_k}\Delta_{k-1}^2
                + \frac{M d^2}{6\beta_K} \sum_{k=1}^{K} \frac{\gamma_k^3}{\beta_k^2}\Delta_{k-1}^3
                + \frac{2B}{R^2 \beta_K} \sum_{k=1}^{K} \beta_{k}r_{k-1}^2.
\end{align}
First, 
  $  \frac{\beta_0}{\beta_K} \leq \frac{\beta_0}{\beta_K-\beta_0} \leq \frac{s^s}{(K/d)^{s+1}} $. Because the sequence $f(x_k)$ is non-increasing, $x_k \in \mathcal{Q}$ for all $k\geq 0$ and so $\Delta_k \leq \mathcal{R}$ (see Definition~\ref{def:R-bounded}).
Using $(1+\frac{1}{K})^s \leq e^{s/K}$, 
\begin{align*}
    \frac{1}{\beta_K}\sum_{k=1}^{K}\frac{\gamma_k^2}{\beta_k}\Delta_{k-1}^2
    & \stackrel{\eqref{eq: beta_k bounds}}{\leq} \frac{\mathcal{R}^2(s+1)^2}{K^{s+1}} \sum_{k=1}^{K}\frac{k^{2s}}{k^{s+1}}
    ~~= \frac{\mathcal{R}^2(s+1)^2}{K^{s+1}} \sum_{k=1}^{K} k^{s-1} \\
    & \leq \frac{\mathcal{R}^2(s+1)^2}{K^{s+1}} \frac{(K+1)^s}{s} 
    ~~\leq \frac{\mathcal{R}^2e^{s/K} (s+1)^2}{sK}
\end{align*}
and 
\begin{align*}
    \frac{1}{\beta_K}\sum_{k=1}^{K}\frac{\gamma_k^3}{\beta_k^2}\Delta_{k-1}^3
    & \stackrel{\eqref{eq: beta_k bounds}}{\leq} \frac{\mathcal{R}^3(s+1)^3}{K^{s+1}} \sum_{k=1}^{K}\frac{k^{3s}}{k^{2s+2}}
    ~~= \frac{\mathcal{R}^3(s+1)^3}{K^{s+1}} \sum_{k=1}^{K} k^{s-2}\\
    &\leq \frac{\mathcal{R}^3(s+1)^3}{K^{s+1}} \frac{(K+1)^{s-1}}{s-1} ~~\leq \frac{\mathcal{R}^3e^{(s-1)/K} (s+1)^3}{(s-1)K^2}.
\end{align*}
Lastly, the error due to the finite difference is controlled by the sampling radius:  
\begin{align*}
    \frac{2B}{R^2 \beta_K}\sum_{k=1}^{K}\beta_{k} r_{k-1}^2
    & \stackrel{\eqref{eq: beta_k bounds}}{\leq} \frac{2(s+1) B\varepsilon \rho^2}{R^2 K^{s+1}} \sum_{k=1}^{K}\frac{(k+1)^{s}}{s+1} + \frac{2B\varepsilon \rho^2 \beta_0}{R^2 \beta_K}\sum_{k=1}^{K} \frac{1}{(k+1)} \\
    & \leq \frac{\varepsilon e^{2(s+1)/K}}{s+1} + \frac{\varepsilon \beta_0 \log(K+2)}{\beta_K}.
\end{align*}
Combining the above with $\varepsilon < \delta_0$ we get
\begin{align}\label{eq: CR conv delta_K bound by four terms}
    \delta_K &\leq \frac{ s^s \delta_0 (1+ \log(K+2)) }{(K/d)^{s+1}}
                + \frac{e^{s/K} (s+1)^2 L\mathcal{R}^2  }{2s(K/d)} \cr
                &\quad~ + \frac{e^{(s-1)/K} (s+1)^3 M \mathcal{R}^3  }{6(s-1)(K/d)^2}
                + \frac{e^{2(s+1)/K} }{s+1}\varepsilon .
\end{align}
When $K>s$, bounding the first three term in \eqref{eq: CR conv delta_K bound by four terms} by $\frac{s-1}{6(s+1)}$, and the last term by $\frac{s+3}{2(s+1)}$ gives the sufficient conditions on $K$:
\begin{equation*}
    \frac{K}{d} \geq \max\left\{
    \frac{3(s+1)^3 e}{s(s-1)}\frac{L\mathcal{R}^2}{\varepsilon},
    \frac{(s+1)^2 \sqrt{e}}{(s-1)} \sqrt{\frac{M\mathcal{R}^3}{\varepsilon}},
    s\left(\frac{6(s+1)}{s-1}\right)^{1/s} \left(\frac{\delta_0}{\varepsilon}\right)^{1/s}
    \right\}
\end{equation*}
and $K \geq \frac{2(s+1)}{\log(1+s/2)}$, respectively.
These immediately give \eqref{eq: CR conv K depending on p}. 
\end{proof}

\section{Experimental Results}
\label{sec:ExperimentalResults}
For a detailed description of all experimental settings and hyperparameters, see Appendix~\ref{appendix: Experiments}. The code for all the experiments can be found online at \url{https://github.com/bumsu-kim/CARS}.

\subsection{Convex Functions}
We compared the performance of CARS and CARS-CR to STP \cite{bergou2020stochastic}, SMTP \cite{gorbunov2019stochastic}, Nesterov-Spokoiny \cite{nesterov2017random}, SPSA \cite{spall1992multivariate}, 2SPSA \cite{spall2000adaptive}, and AdaDGS \cite{tran2020adadgs} 
on the following convex quartic function:
\begin{equation*}
    f(x) = \alpha\sum_{i=1}^{d} x_i^4 + \frac{1}{2}x^{\top}Ax + \beta\|x\|^2,
\end{equation*}
where $\alpha, \beta>0$ and $A = G^{\top}G$ with $G_{ij} \stackrel{i.i.d}{\sim} \mathcal{N}(0, 1)$ for $i, j = 1,2,\cdots,d$. We show in Figure~\ref{fig:Convex Quartic} the objective function value versus the number of function queries.

\begin{figure*}
    \centering
    \includegraphics[width=0.48\linewidth]{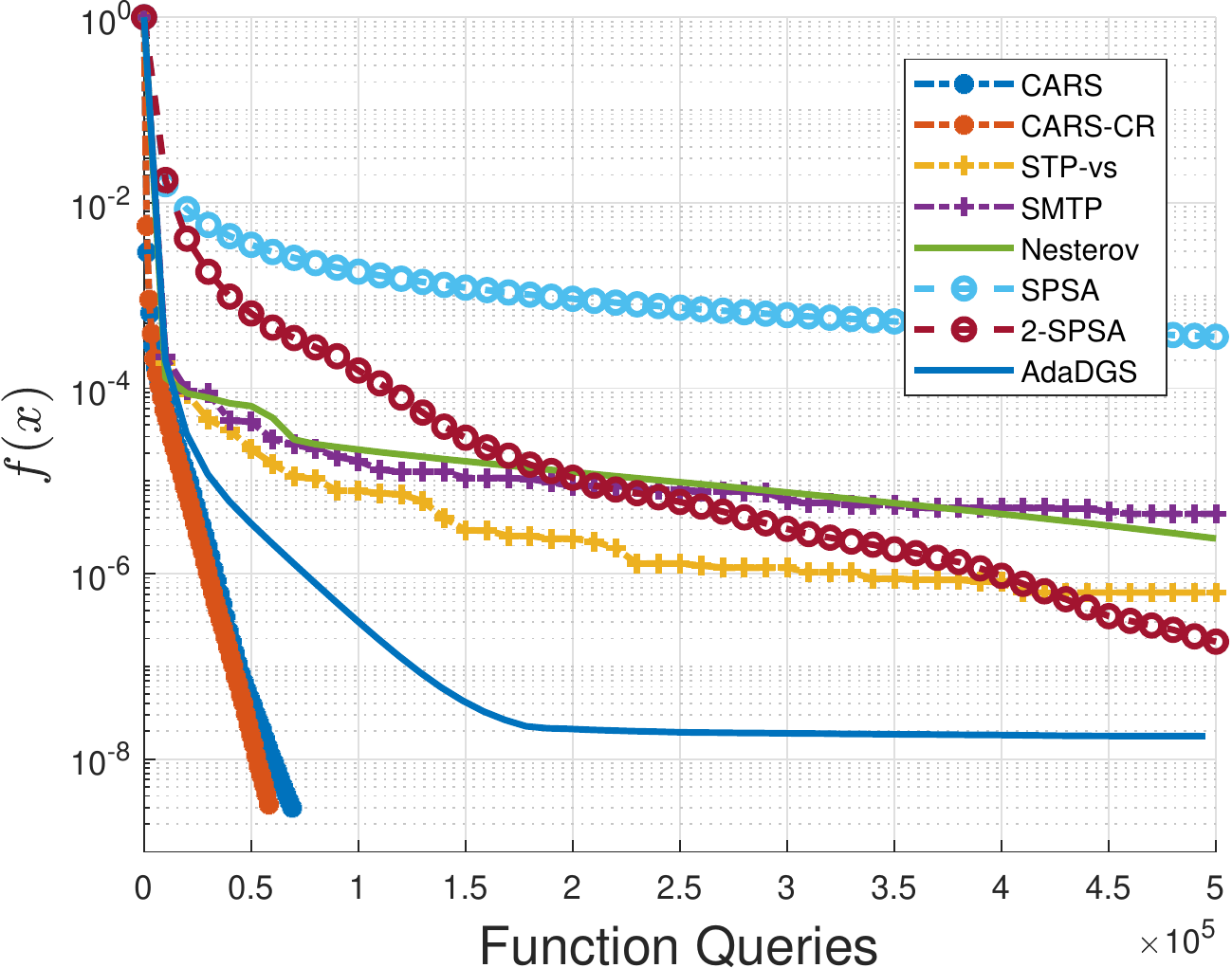}
    \caption{Performance of each algorithm on a convex quartic function $f(x) = 0.1\sum_{i=1}^{d} x_i^4 + \frac{1}{2}x^{\top}Ax + 0.01\|x\|^2$, where $A = G^{\top}G$ with $G_{ij} \stackrel{i.i.d}{\sim} \mathcal{N}(0, 1)$. The problem dimension $d = 30$.}
    \label{fig:Convex Quartic}
\end{figure*}

\subsection{Benchmark Problem Sets with Non-Convex Functions}
The test results in this section are presented in the form of performance profiles \cite{dolan2002benchmarking}, which is a commonly used tool for comparing the performance of multiple algorithms over a suite of test problems. Performance profiles tend to be more informative than single-dimensional summaries ({\em e.g.} average number of iterations required to solve a problem). Formally, consider fixed sets of problems $\mathcal{P}$ and algorithms $\mathcal{S}$. For each $p \in \mathcal{P}$ and $s \in \mathcal{S}$ the {\em performance ratio} $r_{p,s}$ is defined by 
\begin{equation*}
    r_{p,s} = \frac{t_{p,s}}{\min_{s'\in\mathcal{S}} t_{p,s'}},
\end{equation*}
where $t_{p,s}$ is the number of function queries required for $s$ to solve $p$. This is the relative performance of $s$ on $p$ compared to the best algorithm in $\mathcal{S}$ for $p$. The {\em performance profile} of $s$, $\rho_{s} : [1,\infty) \rightarrow [0,1]$ is defined as
\begin{align*}
    \rho_s(\tau) = \frac{|\{p \in \mathcal{P} : r_{p,s} \leq \tau \}|}{|\mathcal{P}|}.
\end{align*}
Therefore, $\rho_s(1)$ is the fraction of problems for which $s$ performs the best, while $\rho_s(\tau)$ for large $\tau$ measures the robustness of $s$. For all $\tau$, a {\em higher value of $\rho_{s}(\tau)$ is better}. We use a log-scale on the horizontal access when plotting $\rho_s(\tau)$.

\vspace{0.1in}
\noindent\textit{\textbf{Mor\'{e}-Garbow-Hillstrom Problems.}}\quad
We tested the same set of algorithms using the well-known non-convex Mor\'{e}-Garbow-Hillstrom 34 test problems \cite{more1981testing}.

For each target accuracy $\varepsilon$, a problem is considered solved when we have $f(x_k)-f_\star \leq \varepsilon(f(x_0)-f_{\star})$ within the budget of 20,000 queries. We used the recommended starting point $x_0$ as in \cite{more1981testing} for all the tested algorithm, and repeated each test 10 times. The results are presented in Figure~\ref{fig:More Garbow Hillstrom and CUTEst}.

\vspace{0.1in}
\noindent\textit{\textbf{CUTEst Problems.}}\quad
We further assessed the performance of CARS and CARS-CR to the same suite of algorithms on the CUTEst \cite{gould2015cutest} problem set, which contains various convex and non-convex problems. 
As before, we compared the methods using performance profiles for the 146 problems with dimension less than or equal to 50. The query budget for each problem was set to be $20,000$ times the problem dimension. The target accuracies were again set to $\varepsilon (f(x_0) - f_\star)$. The results are reported in Figure~\ref{fig:More Garbow Hillstrom and CUTEst}.

\begin{figure*}
    \centering
    \includegraphics[width=0.32\linewidth]{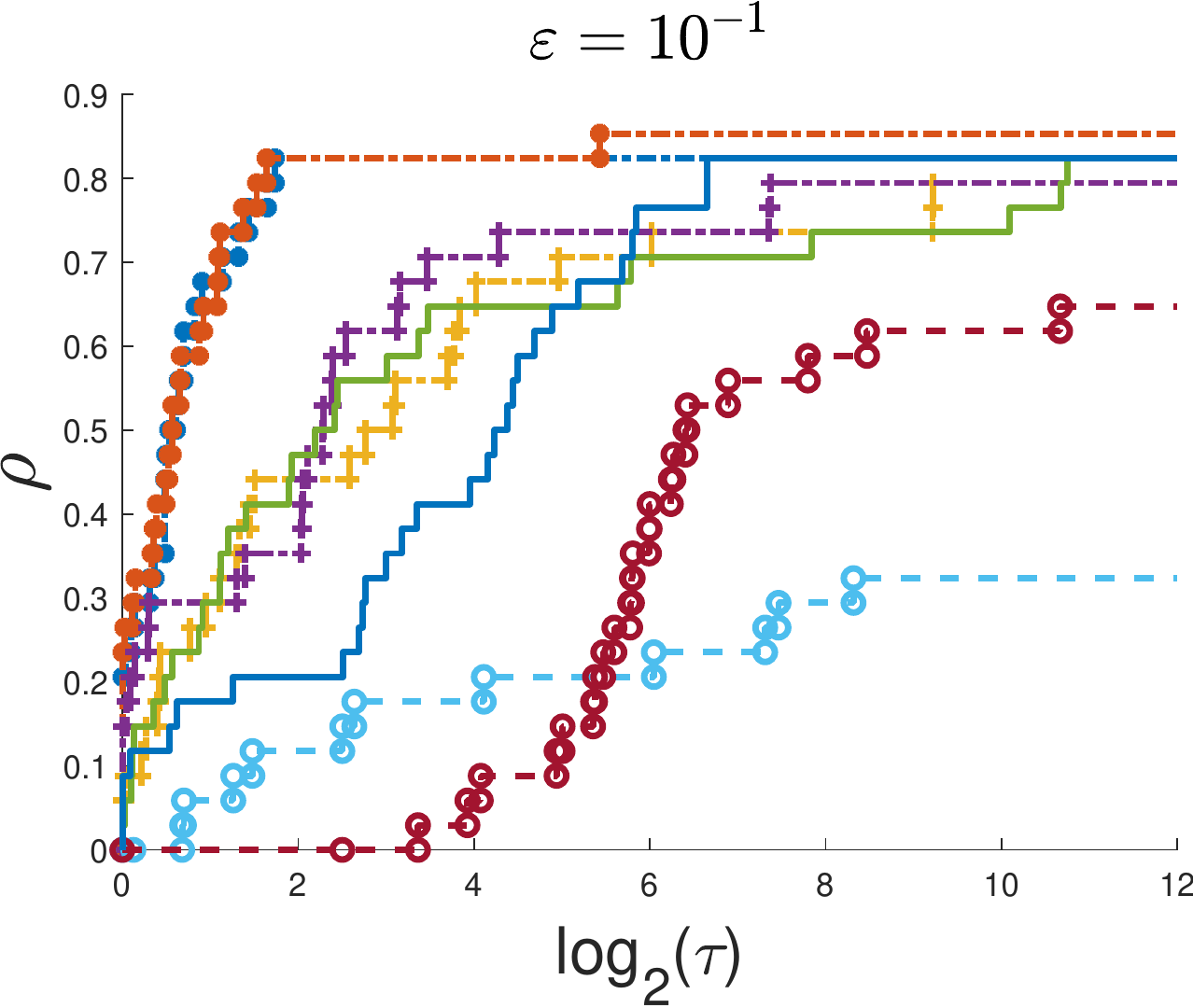}
    \includegraphics[width=0.32\linewidth]{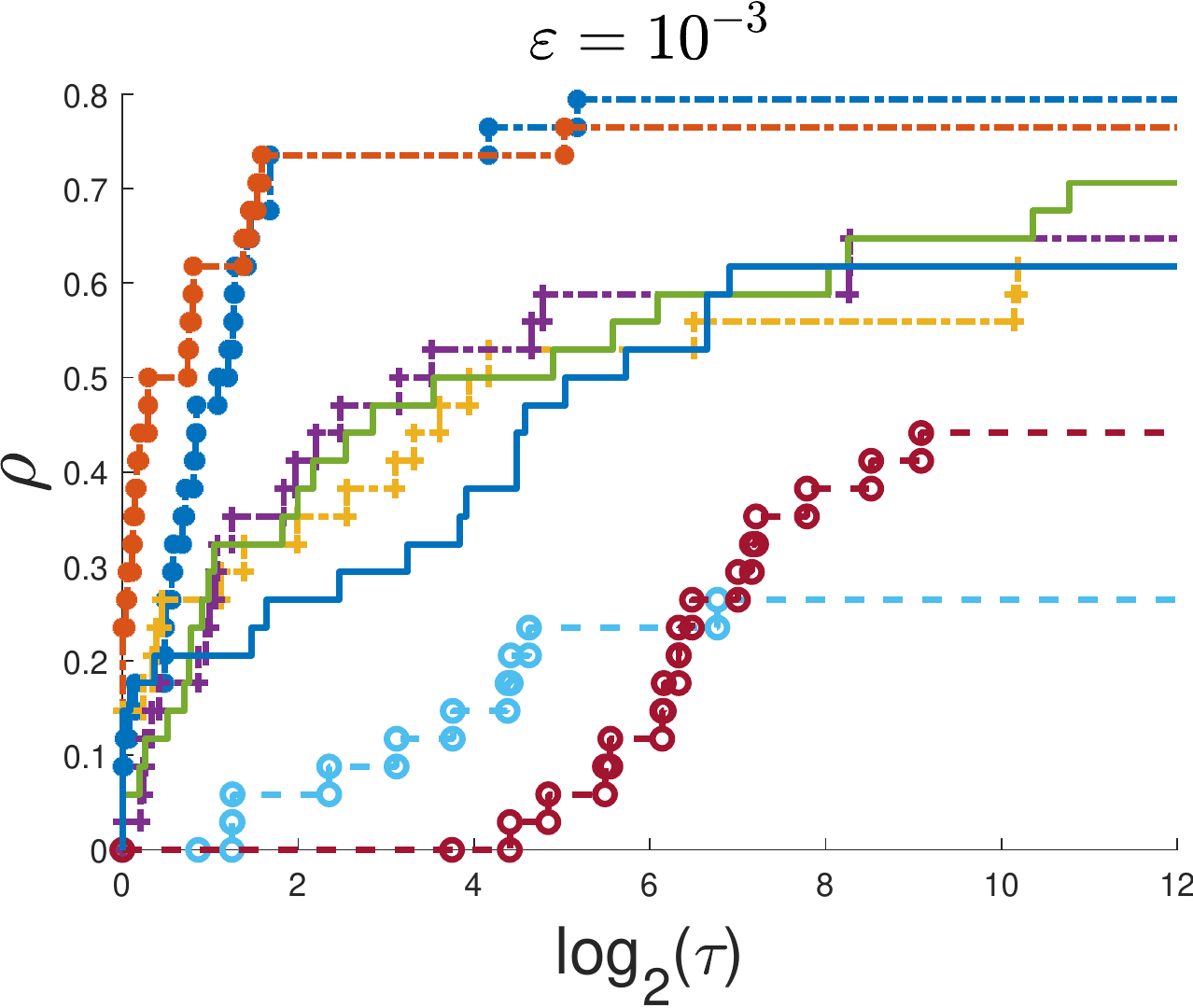}
    \includegraphics[width=0.32\linewidth]{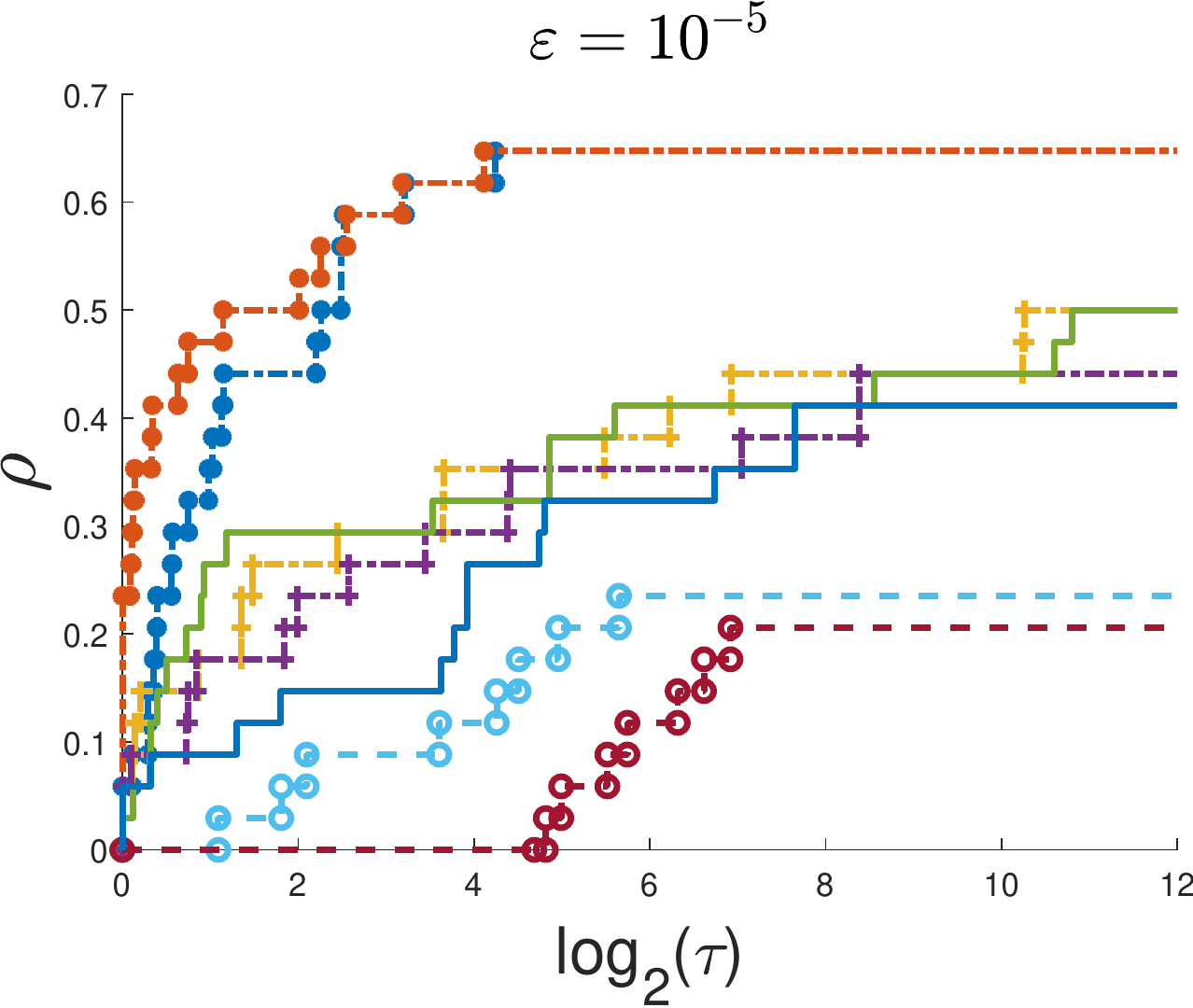}\\
    \vspace{1mm}
    {
      \setlength{\fboxsep}{0pt}
      \fbox{\includegraphics[width=0.8\linewidth]{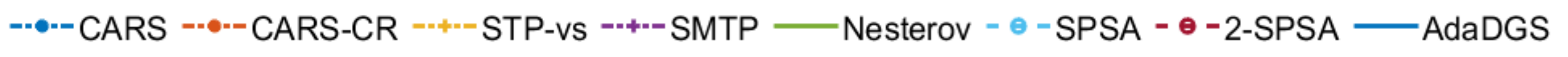}}
    }\\
    \vspace{1mm}
    \includegraphics[width=0.32\linewidth]{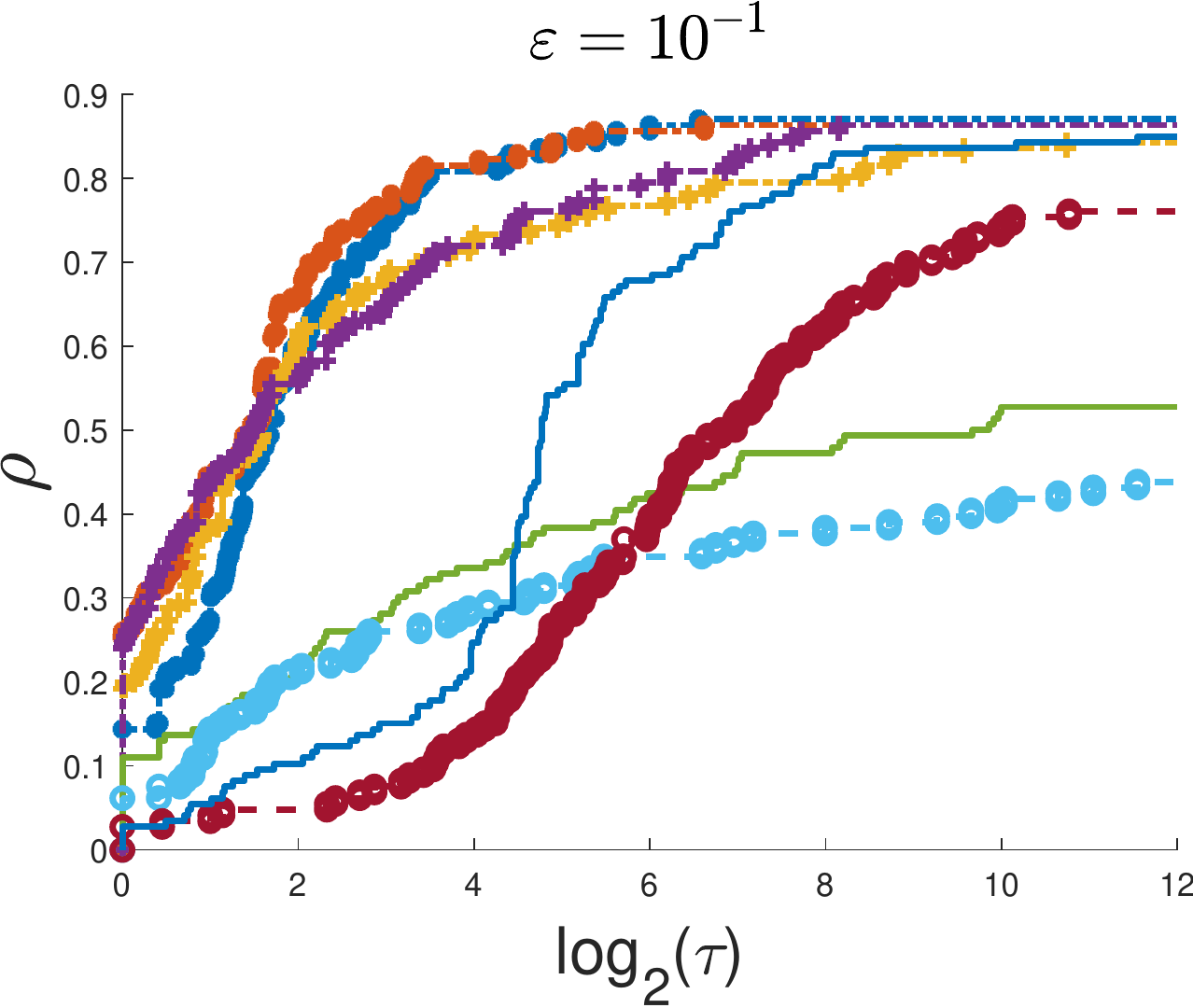}
    \includegraphics[width=0.32\linewidth]{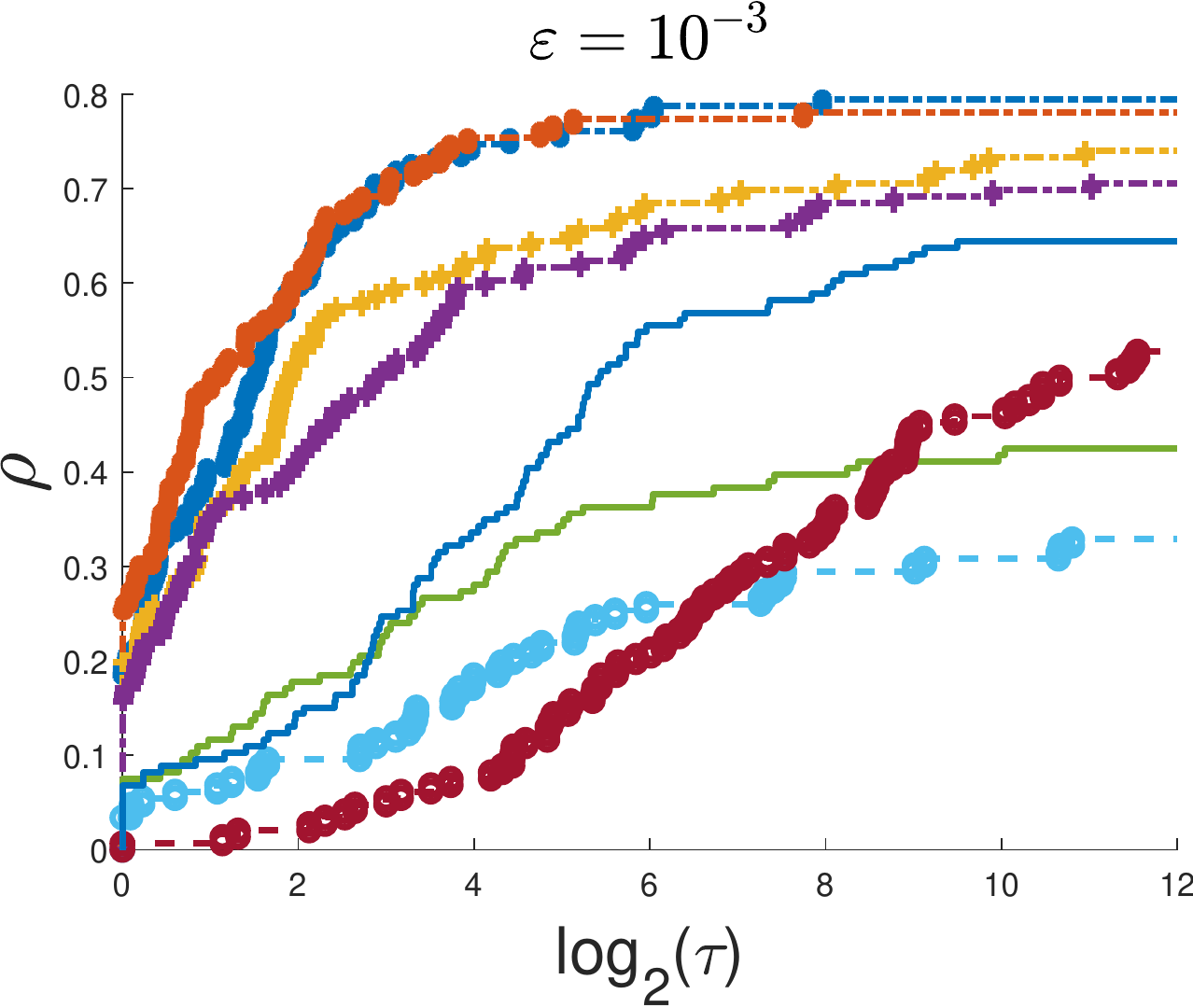}
    \includegraphics[width=0.32\linewidth]{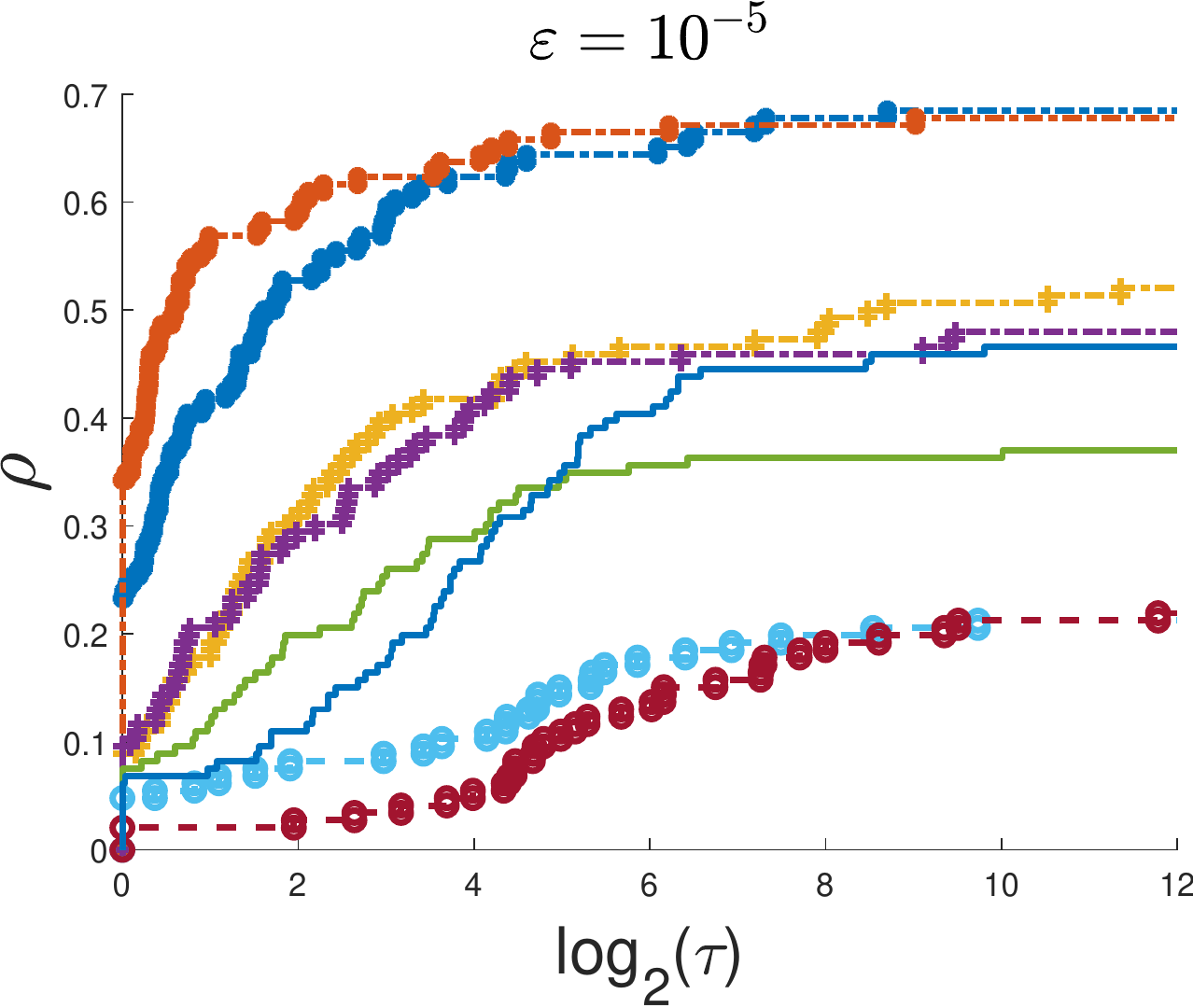}
    \caption{Performance profiles on Mor{\'e}-Garbow-Hillstrom problems (\textbf{upper}) and CUTEst problems (\textbf{lower}), for various target accuracies $\varepsilon = 10^{-1}$ (\textbf{left}), $10^{-3}$ (\textbf{middle}), and $10^{-5}$ (\textbf{right}). Our results demonstrate that CARS and CARS-CR consistently outperform other methods in terms of both efficiency ($\rho$ at low $\tau$ values) and robustness ($\rho$ at high $\tau$ values.) at all levels of accuracy.}\label{fig:More Garbow Hillstrom and CUTEst}
\end{figure*}

\subsection{Black-box Adversarial Attacks}
\begin{table}[t]
\def\ROWCOLOR{black!10!white}
    \begin{tabular}{l c c c}
    \toprule
       Algorithm & Success Rate (\%) & Median Queries & Average Queries \\
    \midrule
        \rowcolor{\ROWCOLOR}
        ZOO$^*$     & 93.95 & 11,700 & 11,804 \\
        PGD-NES$^*$ & 88.39 & 2,450 & 4,584 \\
        \rowcolor{\ROWCOLOR}
        ZOHA-Gauss$^*$ & 91.69 & 1,400 & 2,586 \\
        ZOHA-Diag$^*$ & 91.06 & 1,656  & 3,233 \\
        \rowcolor{\ROWCOLOR}
        STP & 53.64 & 2,193 & 3,141 \\
        SMTP & 65.68 & 1,415 & 2,250 \\
        \rowcolor{\ROWCOLOR}
        Nesterov & 67.72 & 1,105 & 2,044 \\
        Square Attack & {\bf 98.21} & 1,060 & 1,297 \\
        \rowcolor{\ROWCOLOR}
        CARS (Square) & 97.09 & {\bf 717} & {\bf 1,169} \\
    \bottomrule
    \end{tabular}
    \caption{Comparison of success rates, and median and average function queries for the successful black-box adversarial attacks on MNIST with $\ell_\infty$-perturbation bound 0.2. 
    CARS, equipped with the Square Attack's distribution, shows the best performance in successful attacks, while reaching the second best success rate. The results marked with $^*$ are cited from \cite{ye2018hessian}.
    }
    \label{table: Black Box Attack to a CNN model on the MNIST dataset}
\end{table}
Suppose $\mathcal{N}$ is an image classifier.
The problem of generating small perturbations $x$ that, when added to a natural image $x_{\mathrm{nat}}$, fool the classifier ({\em i.e.} $\mathcal{N}(x_{\mathrm{nat}} + x) \neq \mathcal{N}(x_{\mathrm{nat}})$) is known as finding an {\em adversarial attack} \cite{goodfellow2014explaining}. As described in \cite{chen2017zoo}, when no access to the internal workings of the classifier 
is available, this problem becomes a black-box, or derivative-free, optimization problem. In order to ensure the attacked image $x_{\mathrm{nat}} + x$ appears natural, a pixel-wise bound $\|x\|_{\infty} \leq \varepsilon_{\mathrm{atk}}$ is usually enforced. CARS showed state-of-the-art performance in generating black-box adversarial attacks for $\mathcal{N}$ trained on the MNIST digit classification dataset \cite{lecun2010mnist}. 

In our experiments, $\mathcal{N}$ is a two-layer CNN achieving $99\%$ test accuracy on unperturbed images. We use $\varepsilon_{\mathrm{atk}} = 0.2$ and consider all $10,000$ images from the test set of MNIST. We consider an attack a success if it fools $\mathcal{N}$ before a budget of $10,000$ queries is met. The success rates, median and average queries for successful attacks are shown in Table~\ref{table: Black Box Attack to a CNN model on the MNIST dataset}. 
The results from ZOO \cite{chen2017zoo}, PGD-NES \cite{ilyas2018black}, and ZOHA-type algorithms \cite{ye2018hessian} are cited from \cite{ye2018hessian}. 
As pointed out in Section~\ref{section:convergence of CARS}, the choice of sampling directions for CARS is not restrictive. Hence we used a similar initialization and distribution $\mathcal{D}$ as the Square Attack \cite{andriushchenko2020square}, which is known to be particularly well-suited for attacking CNN models.
Visualization of attacked images is partly shown in Figure~\ref{fig:MNIST_ATK_RES}. Detailed settings  can be found in Appendix~\ref{appendix: Experiments}. 

\begin{figure}
    \centering
    \includegraphics[width=0.8\linewidth]{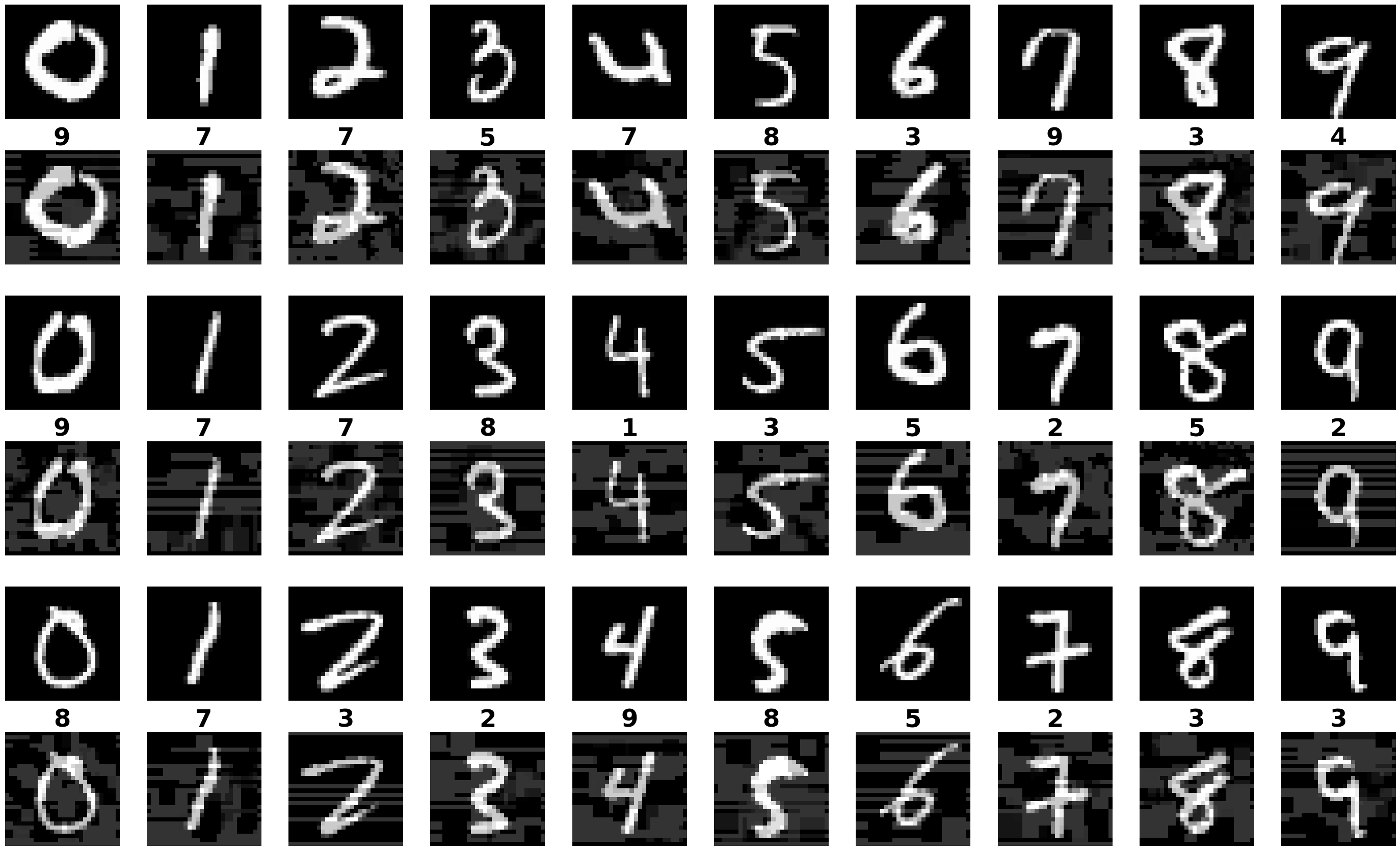}
    \caption{Adversarial examples with misclassified labels on MNIST generated with CARS.} 
    \label{fig:MNIST_ATK_RES}
\end{figure}

\section{Conclusion Remarks} \label{sec:conclusion}
We proposed two query-efficient and lightweight DFO algorithms: CARS and CARS-CR. Our analysis establishes their convergence on strongly convex functions and convex functions. Specifically, we develop a novel and rigorous analysis on the finite difference errors and the probability of significant descents of the objective function. CARS can incorporate various distributions, making it highly adaptable to a range of problem-specific distributions. We demonstrate the efficacy of CARS and CARS-CR through benchmark tests, where it outperforms existing methods in  minimizing non-convex functions as well. 

\section*{Declarations}
\begin{itemize}
\item Funding: The work of HanQin Cai is partially supported by NSF DMS 2304489.
\item Conflict of interest: The authors have no conflicts of interest to declare that are relevant to the content of this article.
\item Code availability: The software code of this paper can be accessed through https://github.com/bumsu-kim/CARS
\end{itemize}

\bibliography{realbib}


\begin{thebibliography}{61}
\ifx \bisbn   \undefined \def \bisbn  #1{ISBN #1}\fi
\ifx \binits  \undefined \def \binits#1{#1}\fi
\ifx \bauthor  \undefined \def \bauthor#1{#1}\fi
\ifx \batitle  \undefined \def \batitle#1{#1}\fi
\ifx \bjtitle  \undefined \def \bjtitle#1{#1}\fi
\ifx \bvolume  \undefined \def \bvolume#1{\textbf{#1}}\fi
\ifx \byear  \undefined \def \byear#1{#1}\fi
\ifx \bissue  \undefined \def \bissue#1{#1}\fi
\ifx \bfpage  \undefined \def \bfpage#1{#1}\fi
\ifx \blpage  \undefined \def \blpage #1{#1}\fi
\ifx \burl  \undefined \def \burl#1{\textsf{#1}}\fi
\ifx \doiurl  \undefined \def \doiurl#1{\url{https://doi.org/#1}}\fi
\ifx \betal  \undefined \def \betal{\textit{et al.}}\fi
\ifx \binstitute  \undefined \def \binstitute#1{#1}\fi
\ifx \binstitutionaled  \undefined \def \binstitutionaled#1{#1}\fi
\ifx \bctitle  \undefined \def \bctitle#1{#1}\fi
\ifx \beditor  \undefined \def \beditor#1{#1}\fi
\ifx \bpublisher  \undefined \def \bpublisher#1{#1}\fi
\ifx \bbtitle  \undefined \def \bbtitle#1{#1}\fi
\ifx \bedition  \undefined \def \bedition#1{#1}\fi
\ifx \bseriesno  \undefined \def \bseriesno#1{#1}\fi
\ifx \blocation  \undefined \def \blocation#1{#1}\fi
\ifx \bsertitle  \undefined \def \bsertitle#1{#1}\fi
\ifx \bsnm \undefined \def \bsnm#1{#1}\fi
\ifx \bsuffix \undefined \def \bsuffix#1{#1}\fi
\ifx \bparticle \undefined \def \bparticle#1{#1}\fi
\ifx \barticle \undefined \def \barticle#1{#1}\fi
\bibcommenthead
\ifx \bconfdate \undefined \def \bconfdate #1{#1}\fi
\ifx \botherref \undefined \def \botherref #1{#1}\fi
\ifx \url \undefined \def \url#1{\textsf{#1}}\fi
\ifx \bchapter \undefined \def \bchapter#1{#1}\fi
\ifx \bbook \undefined \def \bbook#1{#1}\fi
\ifx \bcomment \undefined \def \bcomment#1{#1}\fi
\ifx \oauthor \undefined \def \oauthor#1{#1}\fi
\ifx \citeauthoryear \undefined \def \citeauthoryear#1{#1}\fi
\ifx \endbibitem  \undefined \def \endbibitem {}\fi
\ifx \bconflocation  \undefined \def \bconflocation#1{#1}\fi
\ifx \arxivurl  \undefined \def \arxivurl#1{\textsf{#1}}\fi
\csname PreBibitemsHook\endcsname

\bibitem[\protect\citeauthoryear{Salimans et~al.}{2017}]{salimans2017evolution}
\begin{botherref}
\oauthor{\bsnm{Salimans}, \binits{T.}},
\oauthor{\bsnm{Ho}, \binits{J.}},
\oauthor{\bsnm{Chen}, \binits{X.}},
\oauthor{\bsnm{Sidor}, \binits{S.}},
\oauthor{\bsnm{Sutskever}, \binits{I.}}:
Evolution strategies as a scalable alternative to reinforcement learning.
arXiv preprint arXiv:1703.03864
(2017)
\end{botherref}
\endbibitem

\bibitem[\protect\citeauthoryear{Mania et~al.}{2018}]{mania2018simple}
\begin{bchapter}
\bauthor{\bsnm{Mania}, \binits{H.}},
\bauthor{\bsnm{Guy}, \binits{A.}},
\bauthor{\bsnm{Recht}, \binits{B.}}:
\bctitle{Simple random search of static linear policies is competitive for
  reinforcement learning}.
In: \bbtitle{Proceedings of the 32nd International Conference on Neural
  Information Processing Systems},
pp. \bfpage{1805}--\blpage{1814}
(\byear{2018})
\end{bchapter}
\endbibitem

\bibitem[\protect\citeauthoryear{Choromanski
  et~al.}{2020}]{choromanski2020provably}
\begin{bchapter}
\bauthor{\bsnm{Choromanski}, \binits{K.}},
\bauthor{\bsnm{Pacchiano}, \binits{A.}},
\bauthor{\bsnm{Parker-Holder}, \binits{J.}},
\bauthor{\bsnm{Tang}, \binits{Y.}},
\bauthor{\bsnm{Jain}, \binits{D.}},
\bauthor{\bsnm{Yang}, \binits{Y.}},
\bauthor{\bsnm{Iscen}, \binits{A.}},
\bauthor{\bsnm{Hsu}, \binits{J.}},
\bauthor{\bsnm{Sindhwani}, \binits{V.}}:
\bctitle{Provably robust blackbox optimization for reinforcement learning}.
In: \bbtitle{Conference on Robot Learning},
pp. \bfpage{683}--\blpage{696}
(\byear{2020})
\end{bchapter}
\endbibitem

\bibitem[\protect\citeauthoryear{Bergstra and
  Bengio}{2012}]{bergstra2012random}
\begin{barticle}
\bauthor{\bsnm{Bergstra}, \binits{J.}},
\bauthor{\bsnm{Bengio}, \binits{Y.}}:
\batitle{Random search for hyper-parameter optimization}.
\bjtitle{Journal of Machine Learning Research}
\bvolume{13}(\bissue{1}),
\bfpage{281}--\blpage{305}
(\byear{2012})
\end{barticle}
\endbibitem

\bibitem[\protect\citeauthoryear{Chen et~al.}{2017}]{chen2017zoo}
\begin{bchapter}
\bauthor{\bsnm{Chen}, \binits{P.-Y.}},
\bauthor{\bsnm{Zhang}, \binits{H.}},
\bauthor{\bsnm{Sharma}, \binits{Y.}},
\bauthor{\bsnm{Yi}, \binits{J.}},
\bauthor{\bsnm{Hsieh}, \binits{C.-J.}}:
\bctitle{{ZOO}: Zeroth order optimization based black-box attacks to deep
  neural networks without training substitute models}.
In: \bbtitle{Proceedings of the 10th ACM Workshop on Artificial Intelligence
  and Security},
pp. \bfpage{15}--\blpage{26}
(\byear{2017})
\end{bchapter}
\endbibitem

\bibitem[\protect\citeauthoryear{Cai et~al.}{2022}]{cai2020zeroth}
\begin{barticle}
\bauthor{\bsnm{Cai}, \binits{H.}},
\bauthor{\bsnm{Mckenzie}, \binits{D.}},
\bauthor{\bsnm{Yin}, \binits{W.}},
\bauthor{\bsnm{Zhang}, \binits{Z.}}:
\batitle{Zeroth-order regularized optimization ({ZORO}): Approximately sparse
  gradients and adaptive sampling}.
\bjtitle{SIAM Journal on Optimization}
\bvolume{32}(\bissue{2}),
\bfpage{687}--\blpage{714}
(\byear{2022})
\end{barticle}
\endbibitem

\bibitem[\protect\citeauthoryear{Nelder and Mead}{1965}]{nelder1965simplex}
\begin{barticle}
\bauthor{\bsnm{Nelder}, \binits{J.A.}},
\bauthor{\bsnm{Mead}, \binits{R.}}:
\batitle{A simplex method for function minimization}.
\bjtitle{The computer journal}
\bvolume{7}(\bissue{4}),
\bfpage{308}--\blpage{313}
(\byear{1965})
\end{barticle}
\endbibitem

\bibitem[\protect\citeauthoryear{Kolda et~al.}{2003}]{kolda2003optimization}
\begin{barticle}
\bauthor{\bsnm{Kolda}, \binits{T.G.}},
\bauthor{\bsnm{Lewis}, \binits{R.M.}},
\bauthor{\bsnm{Torczon}, \binits{V.}}:
\batitle{Optimization by direct search: New perspectives on some classical and
  modern methods}.
\bjtitle{SIAM review}
\bvolume{45}(\bissue{3}),
\bfpage{385}--\blpage{482}
(\byear{2003})
\end{barticle}
\endbibitem

\bibitem[\protect\citeauthoryear{Conn et~al.}{2009}]{conn2009introduction}
\begin{bbook}
\bauthor{\bsnm{Conn}, \binits{A.R.}},
\bauthor{\bsnm{Scheinberg}, \binits{K.}},
\bauthor{\bsnm{Vicente}, \binits{L.N.}}:
\bbtitle{Introduction to Derivative-free Optimization},
(\byear{2009})
\end{bbook}
\endbibitem

\bibitem[\protect\citeauthoryear{Cartis and Roberts}{2022}]{cartis2022scalable}
\begin{botherref}
\oauthor{\bsnm{Cartis}, \binits{C.}},
\oauthor{\bsnm{Roberts}, \binits{L.}}:
Scalable subspace methods for derivative-free nonlinear least-squares
  optimization.
Mathematical Programming,
1--64
(2022)
\end{botherref}
\endbibitem

\bibitem[\protect\citeauthoryear{Cartis et~al.}{2022}]{cartis2022global}
\begin{botherref}
\oauthor{\bsnm{Cartis}, \binits{C.}},
\oauthor{\bsnm{Massart}, \binits{E.}},
\oauthor{\bsnm{Otemissov}, \binits{A.}}:
Global optimization using random embeddings.
Mathematical Programming,
1--49
(2022)
\end{botherref}
\endbibitem

\bibitem[\protect\citeauthoryear{Cartis and
  Otemissov}{2022}]{cartis2022dimensionality}
\begin{barticle}
\bauthor{\bsnm{Cartis}, \binits{C.}},
\bauthor{\bsnm{Otemissov}, \binits{A.}}:
\batitle{A dimensionality reduction technique for unconstrained global
  optimization of functions with low effective dimensionality}.
\bjtitle{Information and Inference: A Journal of the IMA}
\bvolume{11}(\bissue{1}),
\bfpage{167}--\blpage{201}
(\byear{2022})
\end{barticle}
\endbibitem

\bibitem[\protect\citeauthoryear{Liu et~al.}{2020}]{liu2020primer}
\begin{barticle}
\bauthor{\bsnm{Liu}, \binits{S.}},
\bauthor{\bsnm{Chen}, \binits{P.-Y.}},
\bauthor{\bsnm{Kailkhura}, \binits{B.}},
\bauthor{\bsnm{Zhang}, \binits{G.}},
\bauthor{\bsnm{Hero~III}, \binits{A.O.}},
\bauthor{\bsnm{Varshney}, \binits{P.K.}}:
\batitle{A primer on zeroth-order optimization in signal processing and machine
  learning: Principals, recent advances, and applications}.
\bjtitle{IEEE Signal Processing Magazine}
\bvolume{37}(\bissue{5}),
\bfpage{43}--\blpage{54}
(\byear{2020})
\end{barticle}
\endbibitem

\bibitem[\protect\citeauthoryear{Berahas
  et~al.}{2021a}]{berahas2021theoretical}
\begin{botherref}
\oauthor{\bsnm{Berahas}, \binits{A.S.}},
\oauthor{\bsnm{Cao}, \binits{L.}},
\oauthor{\bsnm{Choromanski}, \binits{K.}},
\oauthor{\bsnm{Scheinberg}, \binits{K.}}:
A theoretical and empirical comparison of gradient approximations in
  derivative-free optimization.
Foundations of Computational Mathematics,
1--54
(2021)
\end{botherref}
\endbibitem

\bibitem[\protect\citeauthoryear{Berahas et~al.}{2021b}]{berahas2021global}
\begin{barticle}
\bauthor{\bsnm{Berahas}, \binits{A.S.}},
\bauthor{\bsnm{Cao}, \binits{L.}},
\bauthor{\bsnm{Scheinberg}, \binits{K.}}:
\batitle{Global convergence rate analysis of a generic line search algorithm
  with noise}.
\bjtitle{SIAM Journal on Optimization}
\bvolume{31}(\bissue{2}),
\bfpage{1489}--\blpage{1518}
(\byear{2021})
\end{barticle}
\endbibitem

\bibitem[\protect\citeauthoryear{Ghadimi and Lan}{2013}]{ghadimi2013stochastic}
\begin{barticle}
\bauthor{\bsnm{Ghadimi}, \binits{S.}},
\bauthor{\bsnm{Lan}, \binits{G.}}:
\batitle{Stochastic first-and zeroth-order methods for nonconvex stochastic
  programming}.
\bjtitle{SIAM Journal on Optimization}
\bvolume{23}(\bissue{4}),
\bfpage{2341}--\blpage{2368}
(\byear{2013})
\end{barticle}
\endbibitem

\bibitem[\protect\citeauthoryear{Nesterov and
  Spokoiny}{2017}]{nesterov2017random}
\begin{barticle}
\bauthor{\bsnm{Nesterov}, \binits{Y.}},
\bauthor{\bsnm{Spokoiny}, \binits{V.}}:
\batitle{Random gradient-free minimization of convex functions}.
\bjtitle{Foundations of Computational Mathematics}
\bvolume{17}(\bissue{2}),
\bfpage{527}--\blpage{566}
(\byear{2017})
\end{barticle}
\endbibitem

\bibitem[\protect\citeauthoryear{Bergou et~al.}{2020}]{bergou2020stochastic}
\begin{barticle}
\bauthor{\bsnm{Bergou}, \binits{E.H.}},
\bauthor{\bsnm{Gorbunov}, \binits{E.}},
\bauthor{\bsnm{Richtarik}, \binits{P.}}:
\batitle{Stochastic three points method for unconstrained smooth minimization}.
\bjtitle{SIAM Journal on Optimization}
\bvolume{30}(\bissue{4}),
\bfpage{2726}--\blpage{2749}
(\byear{2020})
\end{barticle}
\endbibitem

\bibitem[\protect\citeauthoryear{Hanzely et~al.}{2020}]{pmlr-v119-hanzely20a}
\begin{bchapter}
\bauthor{\bsnm{Hanzely}, \binits{F.}},
\bauthor{\bsnm{Doikov}, \binits{N.}},
\bauthor{\bsnm{Nesterov}, \binits{Y.}},
\bauthor{\bsnm{Richtarik}, \binits{P.}}:
\bctitle{Stochastic subspace cubic {N}ewton method}.
In: \bbtitle{Proceedings of the 37th International Conference on Machine
  Learning},
vol. \bseriesno{119},
pp. \bfpage{4027}--\blpage{4038}
(\byear{2020})
\end{bchapter}
\endbibitem

\bibitem[\protect\citeauthoryear{Gower et~al.}{2019}]{gower2019rsn}
\begin{bchapter}
\bauthor{\bsnm{Gower}, \binits{R.}},
\bauthor{\bsnm{Koralev}, \binits{D.}},
\bauthor{\bsnm{Lieder}, \binits{F.}},
\bauthor{\bsnm{Richt{\'a}rik}, \binits{P.}}:
\bctitle{{RSN}: Randomized subspace newton}.
In: \bbtitle{Advances in Neural Information Processing Systems},
pp. \bfpage{616}--\blpage{625}
(\byear{2019})
\end{bchapter}
\endbibitem

\bibitem[\protect\citeauthoryear{Larson et~al.}{2019}]{larson2019derivative}
\begin{barticle}
\bauthor{\bsnm{Larson}, \binits{J.}},
\bauthor{\bsnm{Menickelly}, \binits{M.}},
\bauthor{\bsnm{Wild}, \binits{S.M.}}:
\batitle{Derivative-free optimization methods}.
\bjtitle{Acta Numerica}
\bvolume{28},
\bfpage{287}--\blpage{404}
(\byear{2019})
\end{barticle}
\endbibitem

\bibitem[\protect\citeauthoryear{Liu et~al.}{2018}]{liu2018zeroth}
\begin{botherref}
\oauthor{\bsnm{Liu}, \binits{S.}},
\oauthor{\bsnm{Kailkhura}, \binits{B.}},
\oauthor{\bsnm{Chen}, \binits{P.-Y.}},
\oauthor{\bsnm{Ting}, \binits{P.}},
\oauthor{\bsnm{Chang}, \binits{S.}},
\oauthor{\bsnm{Amini}, \binits{L.}}:
Zeroth-order stochastic variance reduction for nonconvex optimization.
Advances in Neural Information Processing Systems
\textbf{31}
(2018)
\end{botherref}
\endbibitem

\bibitem[\protect\citeauthoryear{Cheng et~al.}{2019}]{cheng2019sign}
\begin{botherref}
\oauthor{\bsnm{Cheng}, \binits{M.}},
\oauthor{\bsnm{Singh}, \binits{S.}},
\oauthor{\bsnm{Chen}, \binits{P.}},
\oauthor{\bsnm{Chen}, \binits{P.-Y.}},
\oauthor{\bsnm{Liu}, \binits{S.}},
\oauthor{\bsnm{Hsieh}, \binits{C.-J.}}:
Sign-opt: A query-efficient hard-label adversarial attack.
arXiv preprint arXiv:1909.10773
(2019)
\end{botherref}
\endbibitem

\bibitem[\protect\citeauthoryear{Cai et~al.}{2021}]{cai2021zeroth}
\begin{bchapter}
\bauthor{\bsnm{Cai}, \binits{H.}},
\bauthor{\bsnm{Lou}, \binits{Y.}},
\bauthor{\bsnm{McKenzie}, \binits{D.}},
\bauthor{\bsnm{Yin}, \binits{W.}}:
\bctitle{A zeroth-order block coordinate descent algorithm for huge-scale
  black-box optimization}.
In: \bbtitle{Proceedings of the 38th International Conference on Machine
  Learning},
pp. \bfpage{1193}--\blpage{1203}
(\byear{2021}).
\bcomment{PMLR}
\end{bchapter}
\endbibitem

\bibitem[\protect\citeauthoryear{Choromanski
  et~al.}{2018}]{choromanski2018structured}
\begin{bchapter}
\bauthor{\bsnm{Choromanski}, \binits{K.}},
\bauthor{\bsnm{Rowland}, \binits{M.}},
\bauthor{\bsnm{Sindhwani}, \binits{V.}},
\bauthor{\bsnm{Turner}, \binits{R.}},
\bauthor{\bsnm{Weller}, \binits{A.}}:
\bctitle{Structured evolution with compact architectures for scalable policy
  optimization}.
In: \bbtitle{International Conference on Machine Learning},
pp. \bfpage{970}--\blpage{978}
(\byear{2018}).
\bcomment{PMLR}
\end{bchapter}
\endbibitem

\bibitem[\protect\citeauthoryear{Fazel et~al.}{2018}]{fazel2018global}
\begin{bchapter}
\bauthor{\bsnm{Fazel}, \binits{M.}},
\bauthor{\bsnm{Ge}, \binits{R.}},
\bauthor{\bsnm{Kakade}, \binits{S.}},
\bauthor{\bsnm{Mesbahi}, \binits{M.}}:
\bctitle{Global convergence of policy gradient methods for the linear quadratic
  regulator}.
In: \bbtitle{International Conference on Machine Learning},
pp. \bfpage{1467}--\blpage{1476}
(\byear{2018}).
\bcomment{PMLR}
\end{bchapter}
\endbibitem

\bibitem[\protect\citeauthoryear{Nesterov}{2012}]{nesterov2012efficiency}
\begin{barticle}
\bauthor{\bsnm{Nesterov}, \binits{Y.}}:
\batitle{Efficiency of coordinate descent methods on huge-scale optimization
  problems}.
\bjtitle{SIAM Journal on Optimization}
\bvolume{22}(\bissue{2}),
\bfpage{341}--\blpage{362}
(\byear{2012})
\end{barticle}
\endbibitem

\bibitem[\protect\citeauthoryear{Jamieson et~al.}{2012}]{jamieson2012query}
\begin{bchapter}
\bauthor{\bsnm{Jamieson}, \binits{K.G.}},
\bauthor{\bsnm{Nowak}, \binits{R.}},
\bauthor{\bsnm{Recht}, \binits{B.}}:
\bctitle{Query complexity of derivative-free optimization}.
In: \bbtitle{Advances in Neural Information Processing Systems},
vol. \bseriesno{25}
(\byear{2012})
\end{bchapter}
\endbibitem

\bibitem[\protect\citeauthoryear{Karmanov}{1974}]{karmanov1974convergence}
\begin{barticle}
\bauthor{\bsnm{Karmanov}, \binits{V.}}:
\batitle{Convergence estimates for iterative minimization methods}.
\bjtitle{USSR Computational Mathematics and Mathematical Physics}
\bvolume{14}(\bissue{1}),
\bfpage{1}--\blpage{13}
(\byear{1974})
\end{barticle}
\endbibitem

\bibitem[\protect\citeauthoryear{Karmanov}{1975}]{karmanov1975convergence}
\begin{barticle}
\bauthor{\bsnm{Karmanov}, \binits{V.}}:
\batitle{On convergence of a random search method in convex minimization
  problems}.
\bjtitle{Theory of Probability \& Its Applications}
\bvolume{19}(\bissue{4}),
\bfpage{788}--\blpage{794}
(\byear{1975})
\end{barticle}
\endbibitem

\bibitem[\protect\citeauthoryear{Mutsenieks and
  Rastrigin}{1964}]{mutsenieks1964extremal}
\begin{botherref}
\oauthor{\bsnm{Mutsenieks}, \binits{V.}},
\oauthor{\bsnm{Rastrigin}, \binits{L.}}:
Extremal control of continuous multi-parameter systems by the method of random
  search.
Akademiia Nauk SSSR, Izvestiia1, Tekhnichekaia Kibernetika,
101--110
(1964)
\end{botherref}
\endbibitem

\bibitem[\protect\citeauthoryear{Schrack and
  Choit}{1976}]{schrack1976optimized}
\begin{barticle}
\bauthor{\bsnm{Schrack}, \binits{G.}},
\bauthor{\bsnm{Choit}, \binits{M.}}:
\batitle{Optimized relative step size random searches}.
\bjtitle{Mathematical Programming}
\bvolume{10}(\bissue{1}),
\bfpage{230}--\blpage{244}
(\byear{1976})
\end{barticle}
\endbibitem

\bibitem[\protect\citeauthoryear{Krutikov}{1983}]{krutikov1983rate}
\begin{barticle}
\bauthor{\bsnm{Krutikov}, \binits{V.}}:
\batitle{On the rate of convergence of the minimization method along vectors in
  a given directional system}.
\bjtitle{USSR Computational Mathematics and Mathematical Physics}
\bvolume{23}(\bissue{1}),
\bfpage{154}--\blpage{155}
(\byear{1983})
\end{barticle}
\endbibitem

\bibitem[\protect\citeauthoryear{Grippo et~al.}{1988}]{grippo1988global}
\begin{barticle}
\bauthor{\bsnm{Grippo}, \binits{L.}},
\bauthor{\bsnm{Lampariello}, \binits{F.}},
\bauthor{\bsnm{Lucidi}, \binits{S.}}:
\batitle{Global convergence and stabilization of unconstrained minimization
  methods without derivatives}.
\bjtitle{Journal of Optimization Theory and Applications}
\bvolume{56}(\bissue{3}),
\bfpage{385}--\blpage{406}
(\byear{1988})
\end{barticle}
\endbibitem

\bibitem[\protect\citeauthoryear{Grippo and
  Sciandrone}{2007}]{grippo2007nonmonotone}
\begin{barticle}
\bauthor{\bsnm{Grippo}, \binits{L.}},
\bauthor{\bsnm{Sciandrone}, \binits{M.}}:
\batitle{Nonmonotone derivative-free methods for nonlinear equations}.
\bjtitle{Computational Optimization and Applications}
\bvolume{37}(\bissue{3}),
\bfpage{297}--\blpage{328}
(\byear{2007})
\end{barticle}
\endbibitem

\bibitem[\protect\citeauthoryear{Grippo and Rinaldi}{2015}]{grippo2015class}
\begin{barticle}
\bauthor{\bsnm{Grippo}, \binits{L.}},
\bauthor{\bsnm{Rinaldi}, \binits{F.}}:
\batitle{A class of derivative-free nonmonotone optimization algorithms
  employing coordinate rotations and gradient approximations}.
\bjtitle{Computational Optimization and Applications}
\bvolume{60}(\bissue{1}),
\bfpage{1}--\blpage{33}
(\byear{2015})
\end{barticle}
\endbibitem

\bibitem[\protect\citeauthoryear{Stich et~al.}{2013}]{stich2013optimization}
\begin{barticle}
\bauthor{\bsnm{Stich}, \binits{S.U.}},
\bauthor{\bsnm{Muller}, \binits{C.L.}},
\bauthor{\bsnm{Gartner}, \binits{B.}}:
\batitle{Optimization of convex functions with random pursuit}.
\bjtitle{SIAM Journal on Optimization}
\bvolume{23}(\bissue{2}),
\bfpage{1284}--\blpage{1309}
(\byear{2013})
\end{barticle}
\endbibitem

\bibitem[\protect\citeauthoryear{Wang et~al.}{2018}]{wang2018stochastic}
\begin{bchapter}
\bauthor{\bsnm{Wang}, \binits{Y.}},
\bauthor{\bsnm{Du}, \binits{S.}},
\bauthor{\bsnm{Balakrishnan}, \binits{S.}},
\bauthor{\bsnm{Singh}, \binits{A.}}:
\bctitle{Stochastic zeroth-order optimization in high dimensions}.
In: \bbtitle{International Conference on Artificial Intelligence and
  Statistics},
pp. \bfpage{1356}--\blpage{1365}
(\byear{2018}).
\bcomment{PMLR}
\end{bchapter}
\endbibitem

\bibitem[\protect\citeauthoryear{Balasubramanian and
  Ghadimi}{2021}]{balasubramanian2018zeroth}
\begin{botherref}
\oauthor{\bsnm{Balasubramanian}, \binits{K.}},
\oauthor{\bsnm{Ghadimi}, \binits{S.}}:
Zeroth-order nonconvex stochastic optimization: Handling constraints, high
  dimensionality, and saddle points.
Foundations of Computational Mathematics,
1--42
(2021)
\end{botherref}
\endbibitem

\bibitem[\protect\citeauthoryear{Cai et~al.}{2022}]{cai2020scobo}
\begin{barticle}
\bauthor{\bsnm{Cai}, \binits{H.}},
\bauthor{\bsnm{Mckenzie}, \binits{D.}},
\bauthor{\bsnm{Yin}, \binits{W.}},
\bauthor{\bsnm{Zhang}, \binits{Z.}}:
\batitle{A one-bit, comparison-based gradient estimator}.
\bjtitle{Applied and Computational Harmonic Analysis}
\bvolume{60},
\bfpage{242}--\blpage{266}
(\byear{2022})
\end{barticle}
\endbibitem

\bibitem[\protect\citeauthoryear{Cartis and Roberts}{2021}]{cartis2021scalable}
\begin{botherref}
\oauthor{\bsnm{Cartis}, \binits{C.}},
\oauthor{\bsnm{Roberts}, \binits{L.}}:
Scalable subspace methods for derivative-free nonlinear least-squares
  optimization.
arXiv preprint arXiv:2102.12016
(2021)
\end{botherref}
\endbibitem

\bibitem[\protect\citeauthoryear{Bibi et~al.}{2020}]{bibi2020stochastic}
\begin{bchapter}
\bauthor{\bsnm{Bibi}, \binits{A.}},
\bauthor{\bsnm{Bergou}, \binits{E.H.}},
\bauthor{\bsnm{Sener}, \binits{O.}},
\bauthor{\bsnm{Ghanem}, \binits{B.}},
\bauthor{\bsnm{Richtarik}, \binits{P.}}:
\bctitle{A stochastic derivative-free optimization method with importance
  sampling: Theory and learning to control}.
In: \bbtitle{Proceedings of the AAAI Conference on Artificial Intelligence},
vol. \bseriesno{34},
pp. \bfpage{3275}--\blpage{3282}
(\byear{2020})
\end{bchapter}
\endbibitem

\bibitem[\protect\citeauthoryear{Fabian}{1971}]{fabian1971stochastic}
\begin{bchapter}
\bauthor{\bsnm{Fabian}, \binits{V.}}:
\bctitle{Stochastic approximation}.
In: \bbtitle{Optimizing Methods in Statistics},
pp. \bfpage{439}--\blpage{470}
(\byear{1971})
\end{bchapter}
\endbibitem

\bibitem[\protect\citeauthoryear{Spall}{2000}]{spall2000adaptive}
\begin{barticle}
\bauthor{\bsnm{Spall}, \binits{J.C.}}:
\batitle{Adaptive stochastic approximation by the simultaneous perturbation
  method}.
\bjtitle{IEEE Transactions on Automatic Control}
\bvolume{45}(\bissue{10}),
\bfpage{1839}--\blpage{1853}
(\byear{2000})
\end{barticle}
\endbibitem

\bibitem[\protect\citeauthoryear{Ye et~al.}{2018}]{ye2018hessian}
\begin{botherref}
\oauthor{\bsnm{Ye}, \binits{H.}},
\oauthor{\bsnm{Huang}, \binits{Z.}},
\oauthor{\bsnm{Fang}, \binits{C.}},
\oauthor{\bsnm{Li}, \binits{C.J.}},
\oauthor{\bsnm{Zhang}, \binits{T.}}:
Hessian-aware zeroth-order optimization for black-box adversarial attack.
arXiv preprint arXiv:1812.11377
(2018)
\end{botherref}
\endbibitem

\bibitem[\protect\citeauthoryear{Glasmachers and
  Krause}{2020}]{glasmachers2020hessian}
\begin{bchapter}
\bauthor{\bsnm{Glasmachers}, \binits{T.}},
\bauthor{\bsnm{Krause}, \binits{O.}}:
\bctitle{The hessian estimation evolution strategy}.
In: \bbtitle{International Conference on Parallel Problem Solving from Nature},
pp. \bfpage{597}--\blpage{609}
(\byear{2020}).
\bcomment{Springer}
\end{bchapter}
\endbibitem

\bibitem[\protect\citeauthoryear{Zhu et~al.}{2019}]{zhu2019efficient}
\begin{barticle}
\bauthor{\bsnm{Zhu}, \binits{J.}},
\bauthor{\bsnm{Wang}, \binits{L.}},
\bauthor{\bsnm{Spall}, \binits{J.C.}}:
\batitle{Efficient implementation of second-order stochastic approximation
  algorithms in high-dimensional problems}.
\bjtitle{IEEE Transactions on Neural Networks and Learning Systems}
\bvolume{31}(\bissue{8}),
\bfpage{3087}--\blpage{3099}
(\byear{2019})
\end{barticle}
\endbibitem

\bibitem[\protect\citeauthoryear{Zhu}{2020}]{zhu2020hessian}
\begin{botherref}
\oauthor{\bsnm{Zhu}, \binits{J.}}:
Hessian inverse approximation as covariance for random perturbation in
  black-box problems.
arXiv preprint arXiv:2011.13166
(2020)
\end{botherref}
\endbibitem

\bibitem[\protect\citeauthoryear{Mor{\'e} et~al.}{1981}]{more1981testing}
\begin{barticle}
\bauthor{\bsnm{Mor{\'e}}, \binits{J.J.}},
\bauthor{\bsnm{Garbow}, \binits{B.S.}},
\bauthor{\bsnm{Hillstrom}, \binits{K.E.}}:
\batitle{Testing unconstrained optimization software}.
\bjtitle{ACM Transactions on Mathematical Software (TOMS)}
\bvolume{7}(\bissue{1}),
\bfpage{17}--\blpage{41}
(\byear{1981})
\end{barticle}
\endbibitem

\bibitem[\protect\citeauthoryear{Gould et~al.}{2015}]{gould2015cutest}
\begin{barticle}
\bauthor{\bsnm{Gould}, \binits{N.I.}},
\bauthor{\bsnm{Orban}, \binits{D.}},
\bauthor{\bsnm{Toint}, \binits{P.L.}}:
\batitle{Cutest: a constrained and unconstrained testing environment with safe
  threads for mathematical optimization}.
\bjtitle{Computational optimization and applications}
\bvolume{60}(\bissue{3}),
\bfpage{545}--\blpage{557}
(\byear{2015})
\end{barticle}
\endbibitem

\bibitem[\protect\citeauthoryear{Nesterov and Polyak}{2006}]{nesterov2006cubic}
\begin{barticle}
\bauthor{\bsnm{Nesterov}, \binits{Y.}},
\bauthor{\bsnm{Polyak}, \binits{B.T.}}:
\batitle{Cubic regularization of newton method and its global performance}.
\bjtitle{Mathematical Programming}
\bvolume{108}(\bissue{1}),
\bfpage{177}--\blpage{205}
(\byear{2006})
\end{barticle}
\endbibitem

\bibitem[\protect\citeauthoryear{Kozak et~al.}{2021}]{kozak2021stochastic}
\begin{barticle}
\bauthor{\bsnm{Kozak}, \binits{D.}},
\bauthor{\bsnm{Becker}, \binits{S.}},
\bauthor{\bsnm{Doostan}, \binits{A.}},
\bauthor{\bsnm{Tenorio}, \binits{L.}}:
\batitle{A stochastic subspace approach to gradient-free optimization in high
  dimensions}.
\bjtitle{Computational Optimization and Applications}
\bvolume{79}(\bissue{2}),
\bfpage{339}--\blpage{368}
(\byear{2021})
\end{barticle}
\endbibitem

\bibitem[\protect\citeauthoryear{Bertsekas}{1997}]{bertsekas1997nonlinear}
\begin{barticle}
\bauthor{\bsnm{Bertsekas}, \binits{D.P.}}:
\batitle{Nonlinear programming}.
\bjtitle{Journal of the Operational Research Society}
\bvolume{48}(\bissue{3}),
\bfpage{334}--\blpage{334}
(\byear{1997})
\end{barticle}
\endbibitem

\bibitem[\protect\citeauthoryear{Gorbunov
  et~al.}{2019}]{gorbunov2019stochastic}
\begin{botherref}
\oauthor{\bsnm{Gorbunov}, \binits{E.}},
\oauthor{\bsnm{Bibi}, \binits{A.}},
\oauthor{\bsnm{Sener}, \binits{O.}},
\oauthor{\bsnm{Bergou}, \binits{E.H.}},
\oauthor{\bsnm{Richt{\'a}rik}, \binits{P.}}:
A stochastic derivative free optimization method with momentum.
arXiv preprint arXiv:1905.13278
(2019)
\end{botherref}
\endbibitem

\bibitem[\protect\citeauthoryear{Spall et~al.}{1992}]{spall1992multivariate}
\begin{barticle}
\bauthor{\bsnm{Spall}, \binits{J.C.}}, \betal:
\batitle{Multivariate stochastic approximation using a simultaneous
  perturbation gradient approximation}.
\bjtitle{IEEE Transactions on Automatic Control}
\bvolume{37}(\bissue{3}),
\bfpage{332}--\blpage{341}
(\byear{1992})
\end{barticle}
\endbibitem

\bibitem[\protect\citeauthoryear{Tran and Zhang}{2020}]{tran2020adadgs}
\begin{botherref}
\oauthor{\bsnm{Tran}, \binits{H.}},
\oauthor{\bsnm{Zhang}, \binits{G.}}:
{AdaDGS}: An adaptive black-box optimization method with a nonlocal directional
  gaussian smoothing gradient.
arXiv preprint arXiv:2011.02009
(2020)
\end{botherref}
\endbibitem

\bibitem[\protect\citeauthoryear{Dolan and
  Mor{\'e}}{2002}]{dolan2002benchmarking}
\begin{barticle}
\bauthor{\bsnm{Dolan}, \binits{E.D.}},
\bauthor{\bsnm{Mor{\'e}}, \binits{J.J.}}:
\batitle{Benchmarking optimization software with performance profiles}.
\bjtitle{Mathematical Programming}
\bvolume{91}(\bissue{2}),
\bfpage{201}--\blpage{213}
(\byear{2002})
\end{barticle}
\endbibitem

\bibitem[\protect\citeauthoryear{Goodfellow
  et~al.}{2014}]{goodfellow2014explaining}
\begin{botherref}
\oauthor{\bsnm{Goodfellow}, \binits{I.J.}},
\oauthor{\bsnm{Shlens}, \binits{J.}},
\oauthor{\bsnm{Szegedy}, \binits{C.}}:
Explaining and harnessing adversarial examples.
arXiv preprint arXiv:1412.6572
(2014)
\end{botherref}
\endbibitem

\bibitem[\protect\citeauthoryear{LeCun et~al.}{2010}]{lecun2010mnist}
\begin{botherref}
\oauthor{\bsnm{LeCun}, \binits{Y.}},
\oauthor{\bsnm{Cortes}, \binits{C.}},
\oauthor{\bsnm{Burges}, \binits{C.}}:
{MNIST} handwritten digit database.
ATT Labs [Online]. Available: http://yann.lecun.com/exdb/mnist
\textbf{2}
(2010)
\end{botherref}
\endbibitem

\bibitem[\protect\citeauthoryear{Ilyas et~al.}{2018}]{ilyas2018black}
\begin{bchapter}
\bauthor{\bsnm{Ilyas}, \binits{A.}},
\bauthor{\bsnm{Engstrom}, \binits{L.}},
\bauthor{\bsnm{Athalye}, \binits{A.}},
\bauthor{\bsnm{Lin}, \binits{J.}}:
\bctitle{Black-box adversarial attacks with limited queries and information}.
In: \bbtitle{International Conference on Machine Learning},
pp. \bfpage{2137}--\blpage{2146}
(\byear{2018}).
\bcomment{PMLR}
\end{bchapter}
\endbibitem

\bibitem[\protect\citeauthoryear{Andriushchenko
  et~al.}{2020}]{andriushchenko2020square}
\begin{bchapter}
\bauthor{\bsnm{Andriushchenko}, \binits{M.}},
\bauthor{\bsnm{Croce}, \binits{F.}},
\bauthor{\bsnm{Flammarion}, \binits{N.}},
\bauthor{\bsnm{Hein}, \binits{M.}}:
\bctitle{Square attack: a query-efficient black-box adversarial attack via
  random search}.
In: \bbtitle{European Conference on Computer Vision},
pp. \bfpage{484}--\blpage{501}
(\byear{2020}).
\bcomment{Springer}
\end{bchapter}
\endbibitem

\end{thebibliography}

\appendix
\section{More on Numerical Experiments}\label{appendix: Experiments}
In this section, we list the hyperparameters we used for each experiment. The code for all experiments can be found in \url{https://github.com/bumsu-kim/CARS}. We ran experiments on two machines to distribute the load. A laptop equipped with Intel i5-9400F and Nvidia RTX 2060 and a workstation equipped with i9-9940X and two Nvidia RTX 2080 are used.

\vspace{0.1in}
\noindent\textit{\textbf{Mor\'{e}-Garbow-Hillstrom and CUTEst Problems.}}\quad
The Mor\'{e}-Garbow-Hillstrom Problem set consists of 34 non-convex smooth functions, where the problem dimension lies between 2 and 100. This experiment is conducted in Matlab. On the other hand, we used 146 unconstrained problems in the CUTEst Problem set, which have dimension not greater than 50. We used Julia for the CUTEst experiment.

We consider a problem solved when $f(x_k) - f_{\star} \leq \varepsilon(f(x_0) - f_\star)$. The target accuracies used here are $\varepsilon = 10^{-1}, 10^{-3}$ and $10^{-5}$.
For CARS, we used the sampling radius $r_k = 0.5/(k+2)$, $\hat{L} = 2$. For CARS-CR, we used the same sampling radius, and $M = 2$. 
For STP \cite{bergou2020stochastic} and Nesterov-Spokoiny \cite{nesterov2017random} we used the same hyperparameters as given in \cite[Section~8.1]{bergou2020stochastic}. We also used the same decreasing step-size for Stochastic Momentum Three Points method (SMTP) \cite{gorbunov2019stochastic}. For the momentum parameter $\beta$ for SMTP, we followed \cite{gorbunov2019stochastic} and used $\beta = 0.5$.
Namely, following the notations in \cite{bergou2020stochastic} and \cite{gorbunov2019stochastic},
$\mathcal{D} = \unif {(\mathbb{S}^{d-1})}$ and $\alpha_k = \frac{1}{\sqrt{k+1}}$ (STP), $\alpha_k = \frac{1}{4(n+4)}$ and $\mu_k = 10^{-4}$, (Nesterov-Spokoiny), and $\gamma_k = \frac{1}{\sqrt{k+1}}$ and $\beta = 0.5$ (SMTP).
For SPSA \cite{spall1992multivariate} and 2SPSA \cite{spall2000adaptive}, we used the Rademacher distribution ({\em i.e.} $(u_k)_i = \pm 1$ with probability 0.5) for $\mathcal{D}$,  $\alpha = 0.602$, $\gamma = 0.101$, $A = 100$, $a = 0.16$, and $c = 10^{-4}$.
For AdaDGS \cite{tran2020adadgs}, we used the code provided by the authors, by implementing the original Python code in Matlab. Some modifications on hyperparameters are made due to the difference in the scale of problem dimension, and the lack of domain width. First, the original AdaDGS code performs experiments on high dimensional problems  ({\em e.g.} $d=1000$), whereas $2\leq d \leq 100$ in this experiment. Also, the problems are unconstrained, and $\|x_0-x_{\star}\|$ varies from order of $10^0$ to $10^6$.
Thus we used the following modified hyperparameters (following the notation of \cite{tran2020adadgs}):
\begin{enumerate}
    \item The number of points used for line search $S = 100$, since the suggested value $0.05d(M-1)$ is too small for our experiments.
    \item The initial smoothing(sampling) radius $\sigma_0 = 10^{-2}$. We tested $\sigma_0 = 5, 1, 10^{-1}, 10^{-2}$ and $10^{-3}$, and chose the best value. When $\sigma_0 \leq 10^{-1}$ then the results were similar.
\end{enumerate}
For plotting the performance profile, we set the performance ratio $r_{p,s} = r_{M}$ when $p$ is not solved by $s$. Having $r_M = \infty$ is ideal, but setting it by a sufficiently large number does not make any difference. We used $r_M = 10^{20}$.

\vspace{0.1in}
\noindent\textit{\textbf{Black-box Adversarial Attacks.}}\quad
In this section, we explain the experiment setting for black-box adversarial attacks and also provide the hyperparameters that we used.
The CNN model we attack has two $5\times5$ convolutional layers with 6 and 16 output channels, followed by a $4\times 4$ convolutional layer with 120 output channels. Then two fully connected layers with 84 and 10 units follows. Between layers we use ReLU, and between convolutional layers we use $2\times 2$ max-pooling as well. Finally we apply log softmax to the output layer.
The test accuracy of the trained model is 98.99\%.

For this particular experiment, we make three modifications to CARS.
First, since the problem is highly non-convex ($h_r<0$ at around 50\% of the iteration), we do not compute $x_{\mathrm{CARS}}$ when $h_r < 0$ at $k$-th iteration.
The second modification is due to the constraint of the problem. Let $\mathcal{F} = \{x \in [0,1]^d : \|x-x_0\| \leq \varepsilon_{\mathrm{atk}} \}$ denote the feasible set.
Inspired by \cite{andriushchenko2020square}, we also compute $x_{\mathrm{bdry}} = x_k - t_{\mathrm{max}}d_r u_k$, where $t_{\mathrm{max}} = \max \{ t > 0 : x_k - td_r u_k \in \mathcal{F}\}$.
To sum up,
\begin{align*}
    x_{k+1} = 
    \begin{cases}
        \argmin\{f(x_k \pm r_k u_k), f(x_{\mathrm{CARS}}), f(x_{\mathrm{bdry}}) \} & \textrm{ if } h_r>0; \\
        \argmin\{f(x_k \pm r_k u_k), f(x_{\mathrm{bdry}}) \} & \text{ otherwise. }
    \end{cases}
\end{align*}
We use the same sampling distribution as Square Attack \cite{andriushchenko2020square}, which is known to be particularly well-suited for attacking CNN models. 
Lastly, we perturbed $x_0$ by adding horizontal stripes. This choice of initialization is found to be very effective in \cite{andriushchenko2020square}.
\end{document}